\documentclass[reqno]{amsart}

\usepackage{hyperref}
\usepackage{amsmath}
\usepackage{amssymb}
\usepackage{amsfonts}
\usepackage{amsthm}
\usepackage{latexsym}
\usepackage{esint}
\usepackage{mathtools}
\usepackage{mathrsfs}
\usepackage{graphicx}
\usepackage{wrapfig}
\usepackage{color}
\usepackage{bm}
\usepackage[top=1in, bottom=1.25in, left=1.10in, right=1.10in]{geometry}
\usepackage{todonotes}
\usepackage{stackengine}
\allowdisplaybreaks

\newcommand\ocirc[1]{\ensurestackMath{\stackon[1pt]{#1}{\mkern2mu\circ}}}

\newtheorem{theorem}{Theorem}[section]
\newtheorem{lemma}[theorem]{Lemma}
\newtheorem{proposition}[theorem]{Proposition}

\theoremstyle{definition}
\newtheorem{definition}[theorem]{Definition}

\theoremstyle{remark}
\newtheorem{remark}[theorem]{Remark}

\numberwithin{equation}{section}

\newcommand{\bx}{\mathbf{x}}
\newcommand{\by}{{y}}

\newcommand{\bu}{\mathbf{u}}

\newcommand{\bff}{\mathbf{f}}

\newcommand{\bv}{\mathbf{v}}

\newcommand{\bn}{\mathbf{n}}
\newcommand{\be}{\mathbf{e}}
\newcommand{\bd}{\mathbf{d}}
\newcommand{\bT}{\mathbb{T}}

\newcommand{\bfphi}{\boldsymbol{\phi}}

\newcommand{\dy}{\, \mathrm{d}{y}}
\newcommand{\dd}{\,\mathrm{d}}

\newcommand{\dx}{\, \mathrm{d} \mathbf{x}}

\newcommand{\dt}{\, \mathrm{d}t}
\newcommand{\ds}{\, \mathrm{d}s}

\newcommand{\Div}{\mathrm{div}_{\mathbf{x}}}
\newcommand{\divx}{\mathrm{div}_{\mathbf{x}}}

\newcommand{\nabx}{\nabla_{\mathbf{x}}}
\newcommand{\naby}{\partial_y}

\newcommand{\Delx}{\Delta_{\mathbf{x}}}

\newcommand{\R}{\mathbb{R}}

\newcommand{\Oeta}{\Omega_{\eta}}

\newcommand{\bfvarphi}{\bm{\varphi}}

\begin{document}

\title[Flexible Ericksen--Leslie liquid crystal]
{The simplified 2D Ericksen--Leslie liquid crystal model interacting with a 1D flexible shell}



\author{Prince Romeo Mensah}
\address{Faculty of Mathematics,
University of Duisburg-Essen,
Thea-Leymann-Stra{\ss}e 9,
45127 Essen, Germany}


\date{\today}



\begin{abstract}
%
%
We consider the evolution and interaction of a $2$-dimensional nematic liquid crystal of Ericksen--Leslie type  within a $1$-dimensional flexible viscoelastic structure.
This is a fully macroscopic model in which the nematic liquid crystal is modelled by the simplified Ericksen--Leslie system with Ginzburg--Landau approximation. The liquid crystal is contained in a thin viscoelastic shell of arbitrary reference configuration that evolves with respect to the forces exerted by the liquid crystal. 
Barring any degeneracies in the shell, we construct a global weak solution for the coupled system. We then show that any family of such weak solutions that are parametrized by the Ginzburg--Landau coefficient, converges to a weak solution of the original simplified Ericksen--Leslie system without the Ginzburg--Landau term.
\end{abstract}

\maketitle
\section{Introduction}  
The rigorous  mathematical analysis of complex fluid with flexible structures is a new field in the analysis of partial differential equations that arises from fluid mechanics. This typically involve a three-state analysis of the interactions between a solvent, a solute, and a structure, where a mixture of the solvent and the solute evolve within and interact with the structure. 
The resulting multiscale problem require tools from two hitherto separate thriving fields in continuum mechanics, i.e. analysis of complex fluids  and  fluid-structure interaction problems. An initial drawback to merging these two fields comes from their very modeling. Whereas a typical fluid-structure interaction problem will involve the
interplay between the macroscopic description of a solvent and a structure, a complex fluid may deal with the interaction of the macroscopic description of a solvent with a solute on either the microscopic, mesoscopic, or macroscopic scale. 
Moreover, the high molecular weight of most complex
fluids (as compared to a Newtonian or Ideal 
fluids, say) makes it comparatively solid-like and thus affects the surface forces at the boundary.

In trying to further understand these problems and find solutions to them, the rigorous analysis of complex fluid-structure interaction problems was recently initiated in \cite{breit2021incompressible}. Here, the complex fluid is a dilute polymeric fluid consisting of a macroscopic Newtonian solvent and a solute that is described on the mesoscopic scale by a Fokker--Planck equation. This mixture then evolves within and interact with a flexible structure of Koiter type. A weak solution to the underlying system which exists until the Koiter energy degenerates or the flexible shell approaches a self-intersection is then constructed. The corresponding local strong solutions were then constructed in \cite{breit2025existence}. As it turns out, modeling the solute in a polymeric fluid on the macroscopic level offers more analytical advantages as has been shown in \cite{mensah2025equilibration, mensah2025strong, mensah2025vanishing, mensah2025weak}. Ironically,  the construction of weak solutions in the 3D-3D-2D settings for macroscopic 3D solutes suddenly appears out of reach although this is known on the mesoscopic scale.
 
Current technological advancements in flexible liquid crystal display (LCD) phones, computers and televisions have necessitated the analysis of liquid crystals in flexible time-dependent domains. For this reason, we consider in this work, a related three-state problem consisting of the interaction of  a nematic liquid crystal within a thin viscoelastic structure. 
We consider the simplest model in the continuum theory of nematic liquid crystals proposed by Lin \cite{lin1989nonlinear}. Besides being simple, it has the desirable property of retaining the nonlinear structure and other essential mathematical properties of the original Ericksen--Leslie model \cite{ericksen1961conservation, ericksen1962hydrostatic, leslie1979theory}. When the domain is fixed, Lin and Liu analysed the mathematical properties of this simplified model in their seminal work \cite{lin1995nonparabolic}. This has since opened the door for more works \cite{lin2010liquid, lin2000existence, hong2011global, lin2016global, wen2011solutions, xu2012global, huang2014regularity, huang2012blow} on different notions of solutions and several variants of the simplified model.

In keeping up with simplified models that maintain essential physical and mathematical forms, we model the liquid crystal subsystem by the simplified Ericksen--Leslie model by Lin and model the flexible structure by the linear shell equation. These two subsystems are suitably coupled at their interface so that their velocities matches.

With regards to liquid-crystals that are interacting with a body, we only know of three results \cite{binz2024interaction, brandt2025dynamics, geng2024global} all of which  involve  a rigid body immersed in a liquid crystal. Motivated by the formation of liquid crystal colloids, the work \cite{binz2024interaction} based its model off of the original simplified Ericksen--Leslie system with the colloidal particles viewed as rigid bodies. They construct local strong solutions to the underlying problem and also showed that 
this strong solution is global when the initial data is close to constant equilibria. A similar result is then replicated in \cite{brandt2025dynamics} for the general $Q$-tensor model.  These two results are, however, preceded by \cite{geng2024global} where global weak solutions are constructed for the interaction of a rigid body with liquid crystals described by the $Q$-tensor model.
The mathematical analysis of a liquid crystal with a flexible structure is completely open and is the subject of this work.
 
\subsection{Geometric setup and equations of motion}
The reference spatial domain  $\Omega \subset \mathbb{R}^2$ has a boundary $\partial\Omega\subset \mathbb{R}$ that may consist of a flexible part $\omega\subset\mathbb{R}$ and a rigid part $\Gamma\subset \mathbb{R}$. However, because the analysis at the rigid part is essentially trivial when compared to the fixed part, we shall identify the whole of $\partial \Omega$ with $\omega$ and endow it with periodic boundary conditions. Let $I:=(0,T)$ represent a time interval for a given constant $T>0$. The time-dependent  displacement of the structure is given by $\eta:\overline{I}\times\omega\rightarrow(-L,L)$ where $L>0$ is a fixed length of the tubular neighbourhood of $\partial\Omega$ given by
\begin{align*}
S_L:=\{\bx\in \mathbb{R}^2\,:\, \mathrm{dist}(\bx,\partial\Omega
)<L \}.
\end{align*}
For some $k\in\mathbb{N}$ large enough, we now assume that $\partial\Omega$  is parametrized by an injective mapping $\bm{\varphi}\in C^k(\omega;\mathbb{R}^2)$ with $\partial_{y } \bm{\varphi}\neq\bm{0}$.
The set $\partial{\Omega_{\eta}}$, given by 
\begin{align*}
\partial{\Omega_{\eta}}=\big\{\bm{\varphi}_{\eta}:=\bm{\varphi}(\by)+\bn(\by)\eta(t,\by)\, :\, t\in I, \by\in \omega\big\}
\end{align*}
then represents the boundary of the flexible (moving) domain at any instant of time $t\in I$ and with the vector $\bn$,
given by
\begin{align*}
\bn =\frac{\partial_{y}^\perp \bm{\varphi}   }{|\partial_{y } \bm{\varphi}   |}=
\frac{(-\partial_{y} \varphi_2,\partial_{y} \varphi_1)   }{\sqrt{(\partial_{y} \varphi_1)^2+(\partial_{y} \varphi_2)^2} },
\end{align*}
being the outward unit normal at the point $\by\in \omega$. 
We also let $\mathbf{n}_{\eta}:=\bn\circ \bm{\varphi}_{\eta}$ be the corresponding outward unit normal of $\partial{\Omega_{\eta}}$ at the spacetime point $\by\in \omega$ and $t\in I$. Then for $\alpha>0$ sufficiently small where $\alpha<L$, we note that $\bn_{\eta}$ is close to $\bn$ and $\bm{\varphi}_{\eta}$ is close to $\bm{\varphi}$. As a result,  it follows that 
\begin{align}
\label{eq:1705}
|\partial_{y } \bm{\varphi}_{\eta} | \neq0 \quad\text{ and } \quad
\bn\cdot \bn_{\eta}\neq 0
\end{align}
for $\by\in \omega$ and $t\in I$. Thus, in particular, there is no loss of strict positivity of the Jacobian determinant provided that $\Vert \eta\Vert_{L^\infty(I;L^{\infty}(\omega))}<\alpha<L$.

For the interior points, we  transform the  reference domain $\Omega$ into a time-dependent moving domain $\Omega_{\eta}$  whose state at time $t\in\overline{I}$ is given by
\begin{align*}
\Omega_{\eta}
 =\big\{
 \bm{\Psi}_{\eta(t)}(\bx):\, \bx \in \Omega 
  \big\}.
\end{align*}
Here,
\begin{align}
\label{Hanz}
\bm{\Psi}_{\eta(t)}(\bx)=
\begin{cases}
\bx+\bn(\by(\bx))\eta(t,\by(\bx))\phi(s(\bx))     & \quad \text{if } \mathrm{dist}(\bx,\partial\Omega)<L,\\
    \bx & \quad \text{elsewhere } 
  \end{cases}
\end{align}
is the Hanzawa transform with inverse $\bm{\Psi}_{-\eta(t)}$ and where for a point $\bx$ in the neighbourhood of $\partial\Omega$, the vector $\bn(\by(\bx))$ is the outward unit normal at the point $\by(\bx)=\mathrm{arg min}_{\by\in\omega}\vert\bx -\bm{\varphi}(\by)\vert$. Also, $s(\bx)=(\bx-\bm{\varphi}(\by(\bx)))\cdot\bn(\by(\bx))$ and $\phi\in C^\infty(\mathbb{R})$ is a cut-off function that is $\phi\equiv0$ in the neighbourhood of $-L$ and $\phi\equiv1$ in the neighbourhood of $0$. Note that $\bm{\Psi}_{\eta(t)}(\bx)$ can be rewritten as
\begin{align*}
\bm{\Psi}_{\eta(t)}(\bx)=
\begin{cases}
\bm{\varphi}(\by(\bx))+\bn(\by(\bx))[s(\bx)+\eta(t,\by(\bx))\phi(s(\bx)) ]    & \quad \text{if } \mathrm{dist}(\bx,\partial\Omega)<L,\\
    \bx & \quad \text{elsewhere. } 
  \end{cases}
\end{align*}
We also note that for the boundary point $\bm{\varphi}(\by)\in \partial\Omega\subset\R^2$, we have that
\begin{align*}
\big(\bm{\varphi}_\eta\circ\bm{\varphi}^{-1}\big)(\bm{\varphi}(\by))=\bm{\varphi}_\eta(\by)=\bm{\varphi}(\by)+\bn\eta\in \partial\Oeta.
\end{align*}
Therefore, $\bm{\varphi}_\eta\circ\bm{\varphi}^{-1}:\partial\Omega\rightarrow\partial\Oeta$ correspond to the Hanzawa transform $\bm{\Psi}_\eta:\overline{\Omega}\rightarrow\overline{\Oeta}$ restricted to the boundary, i.e., $\bm{\Psi}_\eta\vert_{\partial\Omega}=\bm{\varphi}_\eta\circ\bm{\varphi}^{-1}$. Consequently, in particular,
\begin{align}
\label{bdd3trans}
\bm{\varphi}_\eta=\bm{\Psi}_\eta\circ\bm{\varphi}&\quad \text{on } I \times \omega.
\end{align}
\\
With the above preparation in hand, we consider the simplified Ericksen--Leslie  equation for the flow of nematic liquid crystals  interacting with a flexible structure in  the closure of
\begin{align*}
I\times\Oeta:=\bigcup_{i\in I}\{t\}\times\Omega_{\eta(t)}.
\end{align*}
Our goal is to find a
structure displacement function $\eta:(t, \by)\in I \times \omega \mapsto   \eta(t,\by)\in \mathbb{R}$,  a director field $\bd :(t, \mathbf{x} )\in I \times \Oeta  \mapsto \bd (t, \mathbf{x} ) \in \mathbb{R}^2$,  a fluid velocity field $\mathbf{u}:(t, \mathbf{x})\in I \times \Oeta \mapsto  \mathbf{u}(t, \mathbf{x}) \in \mathbb{R}^2$
 and a pressure function $\pi:(t, \mathbf{x})\in I \times \Oeta \mapsto  \pi(t, \mathbf{x}) \in \mathbb{R}$ 
 such that the system of equations
\begin{align} 
\divx \bu=0
&\qquad \text{in }I \times  \Oeta,
\label{contEq}
\\
\varrho_s\partial_t^2\eta -\gamma\partial_t\naby^2 \eta + \alpha\naby^4 \eta= f(\eta,\bu,\pi,\bd) 
&\qquad \text{on }I \times \omega,
\label{shellEq}
\\
\varrho_f\big(\partial_t \bu  + (\mathbf{u}\cdot \nabx)\mathbf{u} \big)
= 
\divx\mathbb{T}(\bu,\pi,\bd) 
&\qquad \text{in }I \times  \Oeta,
\label{momEq}
\\
\partial_t \bd + (\mathbf{u}\cdot \nabx) \bd
=
\gamma_1   \big( \Delx \bd -\tfrac{1}{\epsilon^2}(\vert \bd\vert^2 - 1)\bd
\big)
&\qquad \text{in }I \times  \Oeta,
\label{directorEq}
\end{align}
are satisfied in the sense of distribution. Here, $\mathbb{T}(\bu,\pi,\bd)$ is the Cauchy stress tensor   given by 
\begin{align}
\label{cauchyStress}
\mathbb{T}(\bu,\pi,\bd)=2\mu\mathbb{D}(\bu) -\pi\mathbb{I}_{2\times2}
- \lambda \nabx \bd\odot \nabx\bd
,\qquad
 \mathbb{D}(\bu)=\frac{1}{2}(\nabx\bu+\nabx\bu^\top).
\end{align}
The parameters   $\gamma,\gamma_1, \epsilon, \varrho_s,  \alpha, \varrho_f, \mu, \lambda  $ are
all strictly positive. However, henceforth, we will set all these parameters to $1$ (except for $\epsilon>0$) since doing so will not affect our analysis.
The force $f$, given by 
\begin{align}
\label{wierdForce}
f(\eta,\bu,\pi,\bd)=-(\mathbb{T}(\bu,\pi,\bd) \bn_\eta)\circ \bm{\varphi}_\eta  \cdot \bn\,|\partial_{y } \bm{\varphi}_{\eta} |
\end{align}
is the surface force that the liquid crystal exerts on the structure with $\bn$ being the  unit normal at $\omega$ and $\bn_\eta$, the  unit normal at $\partial\Oeta$. We  remark that the term in the momentum equation \eqref{momEq} containing the director field can be rewritten as
\begin{align}
\label{dRewritten}
\divx   (\nabx \bd\odot \nabx\bd)
=
\frac{1}{2}
\nabx\vert \nabx\bd\vert^2+  \nabx \bd\Delx\bd
\end{align}
where $\nabx \bd\odot \nabx\bd$ is a $2\times2$ matrix whose $(i,j)$-th entry is $\partial_{x_i}\bd\cdot \partial_{x_j}\bd$ for $1\leq i,j\leq 2$.

The initial conditions for the system \eqref{contEq}-\eqref{directorEq}  are
\begin{align}
&\eta(0, \by) = \eta_0(\by), \quad \partial_t\eta(0, \by) = \eta_*(\by)
&\quad \text{on }  \omega,
\label{initialStructure}
\\
&\bd(0, \bx)=\bd_0(\bx), \qquad\mathbf{u}(0, \cdot) = \mathbf{u}_0(\bx)
&\quad \text{in }  \Omega_{\eta_0}.
\label{initialDensityVelo}
\end{align}
With respect to boundary conditions,  we impose periodicity on $\omega$ 
and further  impose  the following 
\begin{align}
\label{noSlip}
& \bu \circ \bm{\varphi}_\eta =\partial_t\eta\bn
&\quad \text{on }I \times \omega, 
\\&\partial_{\bn_\eta}\bd=0  
&\quad \text{on }I \times \partial\Oeta.\label{dstar}
\end{align} 
Since the fluid is viscous, \eqref{noSlip} illustrates the fact that physically, the fluid adheres to the elastic structure and thus, the velocities will coincide  at the interface. The second condition \eqref{dstar} is the no anchoring or free boundary condition for the director field. It means that the directors are free at the boundary and thus, the director remains uniform with no change occurring near the boundary.

Our system above is a solute-solvent-structure (SSS) three matter state problem of dimension 2D-2D-1D, respectively. The  main reason  why the 3D-3D-2D settings appears out of reach with the current method is the several instances of the application of Ladyzhenskaya's inequality used in the proof. Since this is a dimension-dependent inequality, unfortunately, how it is applied in this work appears to rule out the 3D-3D-2D case. 
 Employing a different technique may be required to remedy this issue in the latter setting. 

\section{Preliminaries and main results}
\noindent 
For any two non-negative quantities $F$ and $G$, we write $F \lesssim G$  if there is a constant $c>0$  such that $F \leq c\,G$. If $F \lesssim G$ and $G\lesssim F$ both hold, we use the notation $F\sim G$.  The scalar matrix product of two matrices $\mathbb{A}=(a_{ij})_{i,j=1}^d$ and $\mathbb{B}=(b_{ij})_{i,j=1}^d$ is denoted by $\mathbb{A}:\mathbb{B}=\sum_{ij}a_{ij}b_{ij}$.
The symbol $\vert \cdot \vert$ may be used in four different contexts. For a scalar function $f\in \mathbb{R}$, $\vert f\vert$ denotes the absolute value of $f$. For a vector $\bff\in \mathbb{R}^d$, $\vert \bff \vert$ denotes the Euclidean norm of $\bff$. For a square matrix $\mathbb{F}\in \mathbb{R}^{d\times d}$, $\vert \mathbb{F} \vert$ shall denote the Frobenius norm $\sqrt{\mathrm{trace}(\mathbb{F}^T\mathbb{F})}$. Also, if $S\subseteq  \mathbb{R}^d$ is  a (sub)set, then $\vert S \vert$ is the $d$-dimensional Lebesgue measure of $S$.
Since we only consider functions on $\omega \subset\mathbb{R}$ with periodic boundary
conditions and zero mean values, we have the following equivalences
\begin{equation}
\begin{aligned}\nonumber
\Vert \cdot\Vert_{W^{2k-1,2}(\omega)}\sim
\Vert \naby^{2k-1}\cdot\Vert_{L^{2}(\omega)},
\qquad
\Vert \cdot\Vert_{W^{2k,2}(\omega)}\sim
\Vert \naby^{2k}\cdot\Vert_{L^{2}(\omega)}, 
\end{aligned} 
\end{equation}
for $k\in\mathbb{N}$. For $I:=(0,T)$, $T>0$, and $\eta\in C(\overline{I}\times\omega)$ satisfying $\|\eta\|_{L^\infty(I\times\omega)}\leq L$ where $L>0$ is a constant, we define for $1\leq p,r\leq\infty$,
\begin{align*} 
L^p(I;L^r(\Omega_\eta))&:=\Big\{v\in L^1(I\times\Omega_\eta):\substack{v(t,\cdot)\in L^r(\Omega_{\eta})\,\,\text{for a.e. }t,\\\|v(t,\cdot)\|_{L^r(\Omega_{\eta})}\in L^p(I)}\Big\},\\
L^p(I;W^{1,r}(\Omega_\eta))&:=\big\{v\in L^p(I;L^r(\Omega_\eta)):\,\,\nabx v\in L^p(I;L^r(\Omega_\eta))\big\}.
\end{align*} 
Higher-order Sobolev spaces can be defined accordingly. For $k>0$ with $k\notin\mathbb N$, we define the fractional Sobolev space $L^p(I;W^{k,r}(\Oeta))$ as the class of $L^p(I;L^r(\Omega_\eta))$-functions $v$ for which 
\begin{align*}
\|v\|_{L^p(I;W^{k,r}(\Oeta))}^p
&=\int_I\bigg(\int_{\Oeta} \vert v\vert^r\dx
+\int_{\Oeta}\int_{\Oeta}\frac{|v(\bx)-v(\bx')|^r}{|\bx-\bx'|^{d+k r}}\dx\dx'\bigg)^{\frac{p}{r}}\dt
\end{align*}
is finite. Accordingly, we can also introduce fractional differentiability in time for the spaces on moving domains.

\subsection{The concept of solutions and the main results}
We now give a precise definition of the notion of a solution we are interested in and state our main result.
\begin{definition}[Weak solution] \label{def:weakSolution}
Let $(\eta_0, \eta_*, \bu_0,\bd_0 )$ be a dataset such that
\begin{equation}
\begin{aligned}
\label{dataset}
&
\eta_0 \in W^{2,2}(\omega) \text{ with } \Vert \eta_0\Vert_{L^\infty(\omega)}<L , 
\\
& 
\eta_* \in L^2(\omega),
\quad 
\bu_0\in L^2_{\mathrm{\Div}}(\Omega_{\eta_0}),
\\&
\text{and }\bu_0\circ\bm{\varphi}_{\eta_0} =\eta_*\bn\text{ on $\omega$}, 
\\&
\bd_0 \in W^{1,2}(\Omega_{\eta_0}),
\quad
\vert \bd_0\vert\leq1 \text{ a.e. in } \Omega_{\eta_0}.
\end{aligned}
\end{equation} 
We call the triple
$(\eta,\bu,\bd)$
a weak solution to the system \eqref{contEq}--\eqref{dstar} with data $(  \eta_0,  \eta_*,\bu_0,\bd_0)$ provided that the following hold:
\begin{itemize}
\item[(a)] The structure displacement $\eta$ satisfies
\begin{align*}
\eta \in W^{1,\infty} \big(I; L^2(\omega) \big)\cap W^{1,2} \big(I; W^{1,2}(\omega) \big)\cap  L^\infty \big(I; W^{2,2}(\omega) \big) 
\quad \text{with} \quad \Vert \eta \Vert_{L^\infty(I \times \omega)} <L,
\end{align*}
as well as $\eta(0,\by)=\eta_0(\by)$ and $\partial_t\eta(0,\by)=\eta_*(\by)$.
\item[(b)] The velocity field $\bu$ satisfies
\begin{align*}
 \bu \in L^\infty \big(I; L^2(\Omega_{\eta}) \big)\cap  L^2 \big(I; W^{1,2}_{\Div}(\Omega_{\eta}) \big) \quad \text{with} \quad 
\bu \circ\bm{\varphi}_\eta = \partial_t\eta\bn\quad\text{on}\quad I\times\omega,
\end{align*}
as well as $\bu(0,\bx)=\bu_0(\bx)$.
\item[(c)] The director field $\bd$ satisfies
\begin{align*}
 \bd \in L^\infty \big(I; W^{1,2}(\Omega_{\eta}) \big)\cap  L^2 \big(I; W^{2,2}(\Omega_{\eta}) \big)
 \quad \text{with} \quad 
\vert \bd\vert\leq1 \quad\text{a.e. in}\quad I\times\Oeta,
\end{align*}
as well as $\bd(0,\bx)=\bd_0(\bx)$.
\item[(d)] For all  $(\phi, {\bfphi}, \bm{\psi}) \in C^\infty(\overline{I}\times\omega) \otimes C^\infty(\overline{I} ;C^\infty_{\Div}(\R^2)) \otimes C^\infty(\overline{I}\times\R^2) $ with $\phi(T,\cdot)= {\bfphi}(T,\cdot)=\bm{\psi}(T,\cdot)=0$ and $\bfphi\circ\bfvarphi_\eta =\phi {\bn}$ on $I\times\omega$, we have
\begin{align*}
\int_I  \frac{\mathrm{d}}{\dt}&\bigg(\int_\omega \partial_t \eta \, \phi \dy
+
\int_{\Oeta}\bu  \cdot {\bfphi}\dx
+
\int_{\Oeta}\bd \cdot \bm{\psi} \dx
\bigg)\dt 
\\
&=
\int_I \int_\omega \big(\partial_t \eta\, \partial_t\phi-\partial_t\naby\eta\,\naby \phi
-
\naby^2\eta\,\naby^2 \phi \big)\dy\dt
\\&+\int_I  \int_{\Oeta}\big(  \bu\cdot \partial_t  {\bfphi} + \bu \otimes \bu: \nabx{\bfphi} 
 -  
\nabx\bu:\nabx {\bfphi} 
+
\nabx \bd\odot \nabx\bd:\nabx \bfphi \big) \dx\dt
\\&
+\int_I  \int_{\Oeta}\big( \bd\cdot \partial_t\bm{\psi} 
+
\bu\otimes \bd:\nabx \bm{\psi}
+
\Delx \bd\cdot  \bm{\psi} -\tfrac{1}{\epsilon^2}( \vert \bd\vert^2 - 1)\bd\cdot  \bm{\psi}
\big) \dx\dt,
\end{align*}
\item[(e)]  
the estimate 
\begin{equation}
\begin{aligned}
\label{energyEst}
 \frac{1}{2}\int_\omega&\big( \vert\partial_t \eta \vert^2+ \vert\naby^2\eta \vert^2\big)(t) \dy
+ 
\frac{1}{2}
\int_{\Oeta} \Big(\vert   \bu  \vert^2
+
\vert\nabx\bd  \vert^2
+
\tfrac{1}{2\epsilon^2}( \vert\bd \vert^2 - 1)^2\Big)(t)
 \dx
\\& 
+
\int_0^t
\int_\omega \vert\partial_{t'}\naby \eta\vert^2\dy \dt'
+
\int_0^t
 \int_{\Omega_{\eta }} \Big( \vert \nabx \bu \vert^2
 +
 \big\vert\Delx\bd - \tfrac{1}{\epsilon^2}(\vert \bd \vert^2 - 1)\bd\big\vert^2 
 \Big)
  \dx\dt' 
\\&\leq
\frac{1}{2}\int_\omega \big( \vert\eta_*\vert^2+ \vert\naby^2\eta_0\vert^2\big) \dy 
+ 
\frac{1}{2}
\int_{\Omega_{\eta_0}} \Big(\vert   \bu_0  \vert^2
+
\vert\nabx\bd_0  \vert^2
+
\tfrac{1}{2\epsilon^2}(\vert\bd_0 \vert^2 - 1)^2\Big) 
 \dx 
\end{aligned}
\end{equation}
holds for almost all $t\in I$. 
\end{itemize}
\end{definition}  
With this definition in hand, we now state our first main result.
\begin{theorem}
\label{thm:main}
For a dataset $( \eta_0, \eta_*, \bu_0,\bd_0 )$ satisfying \eqref{dataset}, a weak solution $(\eta,\bu,\bd)$ exists.
\end{theorem} 
\section{Preparatoy results}
\subsection{A Priori estimate}
\label{sec:apriori}
In this subsection, we give a formal justification of the energy inequality. This can be made rigorous by performing the computations below for the corresponding regularized functions and subsequently passing to the limit in the regularization parameter. Consequently, we assume that $(\eta,\bu, \bd) \in C^\infty(\overline{I} \times \omega)\otimes C^\infty(\overline{I};C^\infty_{\divx}( \Omega_\eta))\otimes C^\infty(\overline{I}\times\Omega_\eta)$ and we consider $(\partial_t\eta, \bu, \Delx\bd-\tfrac{1}{\epsilon^2}(\vert \bd \vert^2 - 1)\bd)$ as test functions for \eqref{shellEq}-\eqref{directorEq}.
If we test \eqref{shellEq} with $\partial_t\eta$ and use \eqref{cauchyStress} and \eqref{wierdForce}, we obtain
\begin{equation}
\begin{aligned}
\label{energyEstC}
\frac{1}{2}
\frac{\dd}{\dt}
\int_\omega \big( \vert\partial_t \eta\vert^2+ \vert\naby^2\eta\vert^2\big) \dy
&+
\int_\omega \vert\partial_t\naby \eta\vert^2\dy 
= 
-
\int_\omega
\partial_t\eta\,
f(\eta,\bu,\pi)\dy 
\\&
= 
-
\int_{\partial\Oeta}
 \mathbb{T}(\bu,\pi,\bd)\bn_\eta\cdot(\partial_t\eta\bn)\circ\bm{\varphi}_\eta^{-1} \dd \mathcal{H}^1.
\end{aligned}
\end{equation}
We now test \eqref{momEq} with $\bu$. First of all, since $\divx\bu=0$, it follows that
\begin{equation}
\begin{aligned}\nonumber
\int_{\Oeta} \Big(\partial_t\bu+ (\mathbf{u}\cdot \nabx)\mathbf{u} 
\Big)\cdot\bu\dx
&=
\frac{1}{2} 
\int_{\Oeta}
\partial_t\vert\bu\vert^2\dx
+
\frac{1}{2} 
\int_{\partial\Oeta}
\vert\bu\vert^2 \bn_\eta\cdot\bu\dd \mathcal{H}^1
\\&
=
\frac{1}{2} \frac{\dd}{\dt}
\int_{\Oeta}
\vert\bu\vert^2\dx
\end{aligned}
\end{equation}
by using Reynolds transport theorem. Using $\divx\bu=0$, we also obtain
\begin{equation}
\begin{aligned}\nonumber
\int_{\Oeta}  
\divx\mathbb{T}(\bu,\pi,\bd)
 \cdot\bu\dx
&=
\int_{\partial\Oeta}
\mathbb{T}(\bu,\pi,\bd)\bn_\eta
 \cdot (\partial_t\eta\bn)\circ\bm{\varphi}_\eta^{-1} \dd \mathcal{H}^1
\\&
-
 \int_{\Oeta}  \vert \nabx \bu \vert^2
  \dx
 +
 \int_{\Oeta}\nabx \bd\odot \nabx\bd:\nabx \bu
  \dx.
\end{aligned}
\end{equation}
It follows that   testing \eqref{momEq} with $\bu$ yields
 \begin{equation}
\begin{aligned}
\label{energyEstA}
\frac{1}{2}
\frac{\dd}{\dt}
\int_{\Oeta} \vert   \bu \vert^2
 \dx
& +
 \int_{\Oeta}  \vert \nabx \bu \vert^2
  \dx
  =
 \int_{\Oeta}\nabx \bd\odot \nabx\bd:\nabx \bu
  \dx 
\\& +
\int_{\partial\Oeta}
\mathbb{T}(\bu,\pi,\bd)\bn_\eta
 \cdot (\partial_t\eta\bn)\circ\bm{\varphi}_\eta^{-1} \dd \mathcal{H}^1.
\end{aligned}
\end{equation}
Our next goal is to  test   \eqref{directorEq} with $\Delx\bd-\tfrac{1}{\epsilon^2}(\vert \bd \vert^2 - 1)\bd$. Before we do this, we first note that by Reynolds transport theorem and \eqref{dstar},
\begin{equation}
\begin{aligned}\nonumber
\int_{\Oeta}  
\partial_t\bd \cdot\Delx\bd \dx
&=
-
\int_{\Oeta}  
\partial_t\nabx\bd :\nabx\bd \dx
+
\int_{\Oeta}  
 \divx(\nabx\bd \partial_t\bd)\dx
\\&=
-\frac{1}{2}
\int_{\Oeta}  
\partial_t\vert\nabx\bd \vert^2 \dx 
\\&=
-\frac{1}{2}
\frac{\dd}{\dt}
\int_{\Oeta}  
\vert\nabx\bd \vert^2 \dx. 
\end{aligned}
\end{equation} 
Similarly, by Reynolds transport theorem,
\begin{equation}
\begin{aligned}\nonumber
\int_{\Oeta} 
\partial_t\bd\cdot \tfrac{1}{\epsilon^2}( \vert \bd \vert^2 - 1)\bd\dx
&=
\frac{1}{2\epsilon^2}
\int_{\Oeta} 
(\partial_t\vert \bd\vert^2 )(\vert \bd \vert^2-1)\dx
\\&= 
\frac{1}{4\epsilon^2}
\int_{\Oeta} 
\partial_t (\vert \bd \vert^2 - 1)^2\dx
\\&
=
\frac{1}{4\epsilon^2}
\frac{\dd}{\dt}
\int_{\Oeta} 
 (\vert \bd \vert^2 -1)^2\dx
-
\frac{1}{4\epsilon^2}
\int_{\partial\Oeta} 
(\partial_t\eta\bn)\circ\bm{\varphi}_\eta^{-1}\cdot\bn_\eta(\vert \bd \vert^2 - 1)^2\dd\mathcal{H}^1.
\end{aligned}
\end{equation} 
Now note that since the identity
\begin{align*}
(\bu\cdot\nabx)\bd\cdot\Delx\bd 
=
\sum u_k\partial_kd_i\partial_j\partial_j d_i 
=
\sum \partial_kd_i\partial_j\partial_j d_i u_k 
=
(\nabx\bd \Delx\bd)\cdot\bu
\end{align*}
hold, it follows from \eqref{dRewritten} and \eqref{dstar} that
\begin{align*}
\int_{\Oeta}  
\big((\mathbf{u}\cdot \nabx)\bd\big) \cdot \Delx\bd\dx
&= 
\int_{\Oeta}  
\Big(
\divx   (\nabx \bd\odot \nabx\bd) 
-
\tfrac{1}{2} 
\nabx\vert \nabx\bd\vert^2
\Big)
\cdot\bu
\dx
\\&=
-
\int_{\Oeta}\nabx \bd\odot \nabx\bd:\nabx \bu
  \dx.
\end{align*} 
Also, note that by Gauss theorem
\begin{equation}
\begin{aligned}\nonumber
\int_{\Oeta}  
\big((\mathbf{u}\cdot \nabx)\bd\big) \cdot \tfrac{1}{\epsilon^2}(\vert \bd \vert^2 - 1)\bd\dx
&=
\frac{1}{2\epsilon^2}
\int_{\Oeta}  
((\mathbf{u}\cdot \nabx)\vert \bd\vert^2  )( \vert \bd \vert^2 -1)\dx
\\&=
\frac{1}{4\epsilon^2}
\int_{\Oeta} 
(\mathbf{u}\cdot \nabx)(\vert \bd\vert^2 -1)^2\dx
\\&= 
\frac{1}{4\epsilon^2}
\int_{\partial\Oeta } 
(\partial_t\eta\bn)\circ\bm{\varphi}_\eta^{-1}\cdot \bn_\eta(\vert \bd \vert^2 - 1)^2\dd \mathcal{H}^1.
\end{aligned}
\end{equation} 
Therefore, testing \eqref{directorEq} with $\Delx\bd- \tfrac{1}{\epsilon^2}(\vert \bd \vert^2 -1)\bd$
 yields
\begin{equation}
\begin{aligned}
\label{energyEstB}
\frac{1}{2}
\frac{\dd}{\dt}&
\int_{\Oeta}  
\vert\nabx\bd \vert^2 \dx
+
\frac{1}{4\epsilon^2}
\frac{\dd}{\dt}
\int_{\Oeta} 
 (\vert \bd \vert^2 - 1)^2\dx
 +\int_{\Oeta} \big\vert\Delx\bd- \tfrac{1}{\epsilon^2}( \vert \bd \vert^2 -1)\bd\big\vert^2  \dx
 \\&=  
-
\int_{\Oeta}\nabx \bd\odot \nabx\bd:\nabx \bu
  \dx.
\end{aligned}
\end{equation}
Adding up \eqref{energyEstC}, \eqref{energyEstA} and \eqref{energyEstB} and integrating the resulting summand in time  yields
\begin{equation}
\begin{aligned}
\label{energyEstD00}
\frac{1}{2}& \frac{\dd}{\dt}
\int_\omega \big( \vert\partial_t \eta \vert^2+ \vert\naby^2\eta \vert^2\big)  \dy
+
\frac{1}{2}\frac{\dd}{\dt}
\int_{\Oeta} \Big(\vert   \bu  \vert^2
+
\vert\nabx\bd  \vert^2
+
\tfrac{1}{2\epsilon^2}(\vert\bd \vert^2 - 1)^2\Big) 
 \dx
\\& 
+ 
\int_\omega \vert\partial_{t}\naby \eta\vert^2\dy  
+ 
 \int_{\Oeta} \Big( \vert \nabx \bu \vert^2
 +
 \big\vert\Delx\bd -  \tfrac{1}{\epsilon^2}(\vert \bd \vert^2 - 1)\bd\big\vert^2 
 \Big)
  \dx 
\\&
=  0.
\end{aligned}
\end{equation} 
Integrating in time now gives the desired energy equality. This equality becomes an inequality when one derives the estimate rigorously from a finite-dimensional approximation and use a lower-semicontinuous argument to pass to the limit in the approximation parameter.

\subsection{Maximum principle for director field}
\label{sec:maxPrin}
This short subsection is devoted to the proof of the  maximum principle enjoyed by the director field in the Ginzburg--Landau approximation of the simplified Ericksen--Leslie model. This is a well-known result but since we don't have a reference for it, we present the short proof here for completeness.
\begin{proposition}
Let $\Omega\subseteq\mathbb{R}^d$ and let
$\bd\in C^1(\overline{I};C^2(\Omega) )$ be a solution of
\begin{align*}
\partial_t \bd + (\mathbf{u}\cdot \nabx) \bd
=
  \Delx \bd -\tfrac{1}{\epsilon^2}(\vert \bd\vert^2 - 1)\bd 
&\qquad \text{in }I \times  \Omega
\end{align*}
with any suitable boundary condition and initial condition
\begin{align*}
&\bd(0, \bx)=\bd_0(\bx),  
&\quad \text{in }  \Omega.
\end{align*}
If $\vert \bd_0\vert\leq1$ in $\Omega $, then $\vert \bd\vert\leq1$  in $I\times\Omega$.
\end{proposition}
\begin{proof}
First of all, we note that if we take the scalar product of \eqref{directorEq} with $2\bd$, we obtain
\begin{align}
\label{vertd2}
\partial_t\vert \bd \vert^2 + (\mathbf{u}\cdot \nabx)\vert  \bd\vert ^2
=
  \Delx \vert \bd\vert ^2-2\vert \nabx\bd\vert ^2 -\tfrac{2}{\epsilon^2}(\vert \bd\vert^2 - 1)\vert \bd\vert ^2
&\qquad \text{in }I \times  \Omega.
\end{align} 
Now, for the sake of a contradiction, let us now assume that there exists a spacetime point $(t_*,\bx_*)\in I \times  \Omega$ such that
\begin{align}\label{great1}
\vert \bd(t_*,\bx_*)\vert^2=\max_{(t,\bx)\in[0,t_*,]\times\Omega }\vert\bd (t,\bx)\vert^2>1.
\end{align}
Then by the second derivative test,
\begin{align*}
\partial_t\vert \bd(t_*,\bx_*)\vert^2=0, \qquad
\nabx\vert \bd(t_*,\bx_*)\vert^2=0, \qquad\Delx\vert \bd(t_*,\bx_*)\vert^2<0.
\end{align*}
It also follow from \eqref{great1} that
\begin{align*}
-\tfrac{2}{\epsilon^2}(\vert \bd(t_*,\bx_*)\vert^2 - 1)\vert \bd(t_*,\bx_*)\vert ^2<0
\end{align*}
whereas the trivial estimate
\begin{align*}
-2\vert \nabx\bd(t_*,\bx_*)\vert ^2\leq 0
\end{align*}
also hold. By using these (in)equalities and defining
\begin{align*}
\bd_*:=\bd(t_*,\bx_*), \qquad \bu_*=\bu(t_*,\bx_*)
\end{align*}
we find that
\begin{align*}
\partial_t\vert \bd_* \vert^2 + (\mathbf{u}_*\cdot \nabx)\vert  \bd_*\vert ^2
=0
\end{align*}
whereas
\begin{align*}
  \Delx \vert \bd_*\vert ^2-2\vert \nabx\bd_*\vert ^2 -\tfrac{2}{\epsilon^2}(\vert \bd_*\vert^2 - 1)\vert \bd_*\vert ^2<0.
\end{align*}
This is a contradiction to the identity \eqref{vertd2} which must hold for all spacetime points.
\end{proof}

\subsection{Equivalent representation for the director field equation}
The following lemma is key in passing to the limit in the Ginzburg-Landau functional. That will be the subject of our second main result. 
\begin{lemma}
\label{appen:lemma}
Let $\Omega\subseteq\mathbb{R}^2$ and $\bd^\perp=(-d_2,d_1)$. 
For $\bd\in C^1(\overline{I};C^2(\Omega) )$ such that $|\bd|=1$ in $I\times\Omega$, the following identities
\begin{enumerate}
\item[(1)] $\Delx\bd+|\nabx \bd|^2\bd
=
\bd^\perp(\Delx\bd\cdot\bd^\perp)=
\bd^\perp\divx((\nabx\bd)\bd^\perp)$;
\item[(2)] $\partial_t \bd =\bd^\perp (\partial_t \bd \cdot\bd^\perp )$;
\item[(2)] $ (\mathbf{u}\cdot \nabx) \bd=\bd^\perp (  (\mathbf{u}\cdot \nabx) \bd \cdot\bd^\perp )$ 
\end{enumerate} 
hold.
\end{lemma}
\begin{proof} To show the first identity in $(1)$, we note that 
\begin{align*}
\Delx\bd+|\nabx \bd|^2\bd=\big(\partial_1\partial_1d_1+\partial_2\partial_2d_1+|\nabx \bd|^2d_1\,,\,\partial_1\partial_1d_2+\partial_2\partial_2d_2+|\nabx \bd|^2d_2 \big)
\end{align*}
where
\begin{align*}
|\nabx \bd|^2=(\partial_1d_1)^2+(\partial_1d_2)^2+(\partial_2d_1)^2+(\partial_2d_2)^2
\end{align*}
is the square of the Frobenius norm.
On the other hand, 
\begin{align*}
\bd^\perp(\Delx\bd\cdot\bd^\perp)
&=
\bd^\perp(-d_2\,\partial_1\partial_1d_1-d_2\,\partial_2\partial_2d_1
+
d_1\,\partial_1\partial_1d_2+d_1\,\partial_2\partial_2d_2)
\\
&=
(d_2^2\,\partial_1\partial_1d_1+d_2^2\,\partial_2\partial_2d_1
-
d_1d_2\,\partial_1\partial_1d_2-d_1d_2\,\partial_2\partial_2d_2,
\\&\qquad\qquad
-d_1d_2\,\partial_1\partial_1d_1-d_1d_2\,\partial_2\partial_2d_1
+
d_1^2\,\partial_1\partial_1d_2+d_1^2\,\partial_2\partial_2d_2)
\\
&=:(e^\perp_1,e^\perp_2).
\end{align*}
Given that $ |\bd|=1$ is equivalent to $d_2^2=1-d_1^2$, it follows that
\begin{align*}
e^\perp_1
&=
\partial_1\partial_1d_1+ \partial_2\partial_2d_1
-
d_1^2\,\partial_1\partial_1d_1-d_1^2\,\partial_2\partial_2d_1
-
d_1d_2\,\partial_1\partial_1d_2-d_1d_2\,\partial_2\partial_2d_2
\\
&=
\partial_1\partial_1d_1+ \partial_2\partial_2d_1
-
d_1\,\partial_1(d_1\partial_1d_1)
+
d_1(\partial_1d_1)^2
-
d_1\,\partial_2(d_1\partial_2d_1)
+
d_1(\partial_2d_1)^2
\\&\qquad -
d_1\,\partial_1(d_2\partial_1d_2)
+
d_1(\partial_1d_2)^2
-
d_1\,\partial_2(d_2\partial_2d_2)
+
d_1(\partial_2d_2)^2
\\
&=
\partial_1\partial_1d_1+ \partial_2\partial_2d_1
+|\nabx \bd|^2d_1
-
d_1\,\partial_1(d_1\partial_1d_1) 
-
d_1\,\partial_2(d_1\partial_2d_1) 
\\&\qquad -
d_1\,\partial_1(d_2\partial_1d_2) 
-
d_1\,\partial_2(d_2\partial_2d_2) 
\\
&=
\partial_1\partial_1d_1+ \partial_2\partial_2d_1
+|\nabx \bd|^2d_1
-
\tfrac{1}{2}
d_1\big[\partial_1\partial_1d_1^2
+
\partial_2 \partial_2d_1^2
+
\partial_1\partial_1d_2^2 
+
\partial_2\partial_2d_2^2 \big]
\end{align*}
where using $d_1^2+d_2^2=1$ result in
\begin{align*}
\partial_1\partial_1d_1^2
+
\partial_2 \partial_2d_1^2
+
\partial_1\partial_1d_2^2 
+
\partial_2\partial_2d_2^2
=
\partial_1\partial_11
+
\partial_2 \partial_21
=0.
\end{align*}
Thus
\begin{align*}
e^\perp_1
=
\partial_1\partial_1d_1+ \partial_2\partial_2d_1
+|\nabx \bd|^2d_1 ,
\end{align*}
and similarly
\begin{align*}
e^\perp_2
=
\partial_1\partial_1d_2+ \partial_2\partial_2d_2
+|\nabx \bd|^2d_2.
\end{align*}
The desired first identity in $(1)$  follow.

To show the second identity in $(1)$, we note that 
\begin{align*}
\divx((\nabx\bd)\bd^\perp)
&=
\divx 
\begin{pmatrix}
-d_2\,\partial_1d_1+d_1\,\partial_1d_2\\
-d_2\,\partial_2d_1+d_1\,\partial_2d_2
\end{pmatrix} 
\\
&=
-\partial_1d_2\,\partial_1d_1-d_2\,\partial_1\partial_1d_1+\partial_1d_1\,\partial_1d_2+d_1\,\partial_1\partial_1d_2
\\&\qquad
-\partial_2d_2\,\partial_2d_1-d_2\,\partial_2\partial_2d_1+\partial_2d_1\,\partial_2d_2+d_1\,\partial_2\partial_2d_2 
\\
&=
-d_2\,\partial_1\partial_1d_1+d_1\,\partial_1\partial_1d_2
-d_2\,\partial_2\partial_2d_1+d_1\,\partial_2\partial_2d_2
\\&=\Delx\bd\cdot\bd^\perp.
\end{align*}
The desired second identity in $(1)$ easily follow.

To show $(2)$, we note that
\begin{align*}
\bd^\perp \partial_t \bd \cdot\bd^\perp
=&
(-d_2\,,\,d_1)(-d_2 \partial_td_1+d_1 \partial_td_2)
\\=&
(d_2^2 \partial_td_1-d_1d_2 \partial_td_2\,,\,-d_1d_2 \partial_td_1+d_1^2 \partial_td_2)
\\=&
\big((1-d_1^2) \partial_td_1-\tfrac{1}{2} d_1  \partial_td_2^2\,,\,-\tfrac{1}{2}d_2 \partial_td_1^2+(1-d_2^2) \partial_td_2\big)
\\=&
\big(  \partial_td_1-\tfrac{1}{2} d_1  \partial_td_1^2-\tfrac{1}{2} d_1  \partial_td_2^2\,,\,-\tfrac{1}{2}d_2 \partial_td_1^2-\tfrac{1}{2}d_2 \partial_td_2^2+ \partial_td_2\big)
\\=&
\big(  \partial_td_1-\tfrac{1}{2} d_1  \partial_t(d_1^2- d_2^2)\,,\,-\tfrac{1}{2}d_2 \partial_t(d_1^2-d_2^2)+ \partial_td_2\big)
\\=&
\big(  \partial_td_1-\tfrac{1}{2} d_1  \partial_t(1)\,,\,-\tfrac{1}{2}d_2 \partial_t(1)+ \partial_td_2\big)
\\=&\partial_t\bd.
\end{align*}
Similarly, for $(3)$,
\begin{align*}
\bd^\perp\big( (\mathbf{u}\cdot \nabx) \bd \cdot\bd^\perp\big)
=&
(-d_2\,,\,d_1)(-u_1d_2 \partial_1d_1
-
u_2d_2 \partial_2d_1
+u_1d_1 \partial_1d_2+u_2d_1 \partial_2d_2)
\\
=&
 (u_1d_2^2 \partial_1d_1
+
u_2d_2^2 \partial_2d_1
-u_1d_1d_2 \partial_1d_2-u_2d_1d_2 \partial_2d_2\,,
\\&\qquad
-u_1d_1d_2 \partial_1d_1
-
u_2d_1d_2 \partial_2d_1
+u_1d_1^2 \partial_1d_2+u_2d_1^2 \partial_2d_2
)
\\
=&
 \big(u_1  \partial_1d_1
 -
 \tfrac{1}{2}
 u_1d_1 \partial_1d_1^2
+
u_2  \partial_2d_1
-
 \tfrac{1}{2}
u_2d_1 \partial_2d_1^2
- \tfrac{1}{2}u_1d_1  \partial_1d_2^2- \tfrac{1}{2}u_2d_1  \partial_2d_2^2\,,
\\&\qquad
- \tfrac{1}{2}u_1 d_2 \partial_1d_1^2
-
 \tfrac{1}{2}u_2 d_2 \partial_2d_1^2
+u_1  \partial_1d_2
-
\tfrac{1}{2}u_1d_2 \partial_1d_2^2
+u_2  \partial_2d_2
-
\tfrac{1}{2}u_2d_2\partial_2d_2^2
\big)
\\
=&
 \big(u_1  \partial_1d_1
 -
 \tfrac{1}{2}
 u_1d_1 \partial_1(1)
+
u_2  \partial_2d_1
-
 \tfrac{1}{2}
u_2d_1 \partial_2(1) 
\,,
\\&\qquad
- \tfrac{1}{2}u_1 d_2 \partial_1(1)
-
 \tfrac{1}{2}u_2 d_2 \partial_2(1)
+u_1  \partial_1d_2 
+u_2  \partial_2d_2 
\big)
\\
=&
 \big(u_1  \partial_1d_1 
+
u_2  \partial_2d_1 
\,,\,  
+u_1  \partial_1d_2 
+u_2  \partial_2d_2 
\big)
\\
=&(\mathbf{u}\cdot \nabx) \bd.
\end{align*}
This finishes the proof.
\end{proof}

\section{Proof of Theorem \ref{thm:main}}
Our first goal is to construct a weak solution to a system with a regularized geometry and a linearized advection term. Here, by a  \textit{regularized geometry}, we mean a regularization  of a solution to a \textit{given} shell equation and not the solution to our anticipated shell equation \eqref{shellEq}. Also, we linearized our system by replacing the advecting velocity field with a regularization  of a \textit{given} velocity field. Note that we use the term `linearized' loosely since there remain nonlinear terms with respect to the director field.
This solution we seek will be constructed as the limit $N\rightarrow\infty$ of a finite dimensional Galerkin approximation system incorporating these regularizing terms.

\subsection{Galerkin approximation for linearized system} 
\noindent 

We follow the approach of \cite{lengeler2014weak} and consider a  given structure displacement $\zeta\in C(\overline{I}\times \omega)$  with an initial state $\zeta(0,\cdot)=\eta_0$ and a given driving velocity field $\bv\in L^2_{\divx}(I\times \mathbb{R}^2)$. To ensure that the pair $(\zeta,\bv)$ are sufficiently smooth so that the subsequent analysis are well defined, we consider their spatial regularisation $(\zeta_\kappa,\bv_\kappa)_{\kappa>0}$\footnote{Here, $f_\kappa:=\mathcal{R}_\kappa f$, where the regularising kernels $(\mathcal{R}_\kappa)_{\kappa>0}$ commutes with $\partial_t$. See \cite{lengeler2014weak} for more details.} and assume that they satisfy the interface condition $ \bv_\kappa\circ\bm{\varphi}_{\zeta_\kappa}=\bn\partial_t \zeta_\kappa$ on $I\times\omega$.
Now, for a fixed $\kappa>0$, we seek to construct a triple
$(\eta_\kappa,\bu_\kappa,\bd_\kappa)$ that is a  weak solution of the following system posed on a given regularized geometry, i.e. 
\begin{align}
\divx \bu_\kappa=0
&\qquad \text{in }I \times \Omega_{\zeta_\kappa},
\label{contEqKappaLin}
\\
\partial_t^2\eta_\kappa - \partial_t\naby^2 \eta_\kappa +  \naby^4\eta_\kappa
=
 f(\zeta_\kappa,\bu_\kappa,\pi_\kappa,\bd_\kappa)
&\qquad \text{on }I \times \omega,
\label{shellEqKappaLin}
\\
\partial_t \bu_\kappa  + (\bv_\kappa\cdot \nabx)\bu_\kappa
= 
\divx\mathbb{T}(\bu_\kappa,\pi_\kappa,\bd_\kappa) 
&\qquad \text{in }I \times  \Omega_{\zeta_\kappa},
\label{momEqKappaLin}
\\
\partial_t \bd_\kappa + (\bu_\kappa\cdot \nabx) \bd_\kappa
=
  \Delx \bd_\kappa - \tfrac{1}{\epsilon^2}(\vert \bd_\kappa \vert^2 -1)\bd_\kappa 
&\qquad \text{in }I \times  \Omega_{\zeta_\kappa},
\label{directorEqKappaLin}
\end{align}
where
\begin{align*}
&f(\zeta_\kappa,\bu_\kappa,\pi_\kappa,\bd_\kappa)
=
-(\mathbb{T}(\bu_\kappa, \pi_\kappa,\bd_\kappa )\bn_{\zeta_\kappa }   )\circ\bm{\varphi}_{\zeta_\kappa }
\cdot
\bn 
\,|\partial_{y } \bm{\varphi}_{\zeta_\kappa } | ,
\\&\mathbb{T}(\bu_\kappa, \pi_\kappa,\bd_\kappa )= (\nabx\bu_\kappa + \nabx(\bu_\kappa)^\top)-\pi_\kappa\mathbb{I}_{2\times2}
-  \nabx\bd_\kappa\odot \nabx\bd_\kappa.
\end{align*} 
We complement \eqref{contEqKappaLin}-\eqref{directorEqKappaLin} with the following initial and boundary conditions
\begin{align}
&\eta_\kappa(0, \by) = \eta_{\kappa,0}(\by), \quad \partial_t\eta_\kappa(0, \by) = \eta_{\kappa,*}(\by)
&\quad \text{on }  \omega, 
\\
&\bd_\kappa(0, \bx)=\bd_{\kappa,0}(\bx), \qquad\mathbf{u}_\kappa(0, \cdot) = \mathbf{u}_{\kappa,0}(\bx)
&\quad \text{in }  \Omega_{\eta_0},
\\&\partial_{\bn_{\zeta_\kappa}}\bd_\kappa(t, \bx)=0  
&\quad \text{on }I \times \partial\Omega_{\zeta_\kappa}.
\end{align}  
The dataset $(\eta_{\kappa,0}, \eta_{\kappa,*}, \bu_{\kappa,0},\bd_{\kappa,0} )$ are suitably constructed to satisfy an analogue of the property of the original dataset in \eqref{dataset} with strong convergence in their corresponding spaces. More precisely, we first build a smooth set of initial shell displacements $(\eta_{\kappa,0})_{\kappa>0}$ such that
\begin{align*}
\eta_{\kappa,0}\rightarrow \eta_0 &\qquad\text{in}\qquad   W^{2,2}(\omega)
\end{align*}
as $\kappa\rightarrow0$. Such as regularisation can be done by convolution with the standard mollifier.

Next, we note that since $\eta_*\in L^2(\omega)$ and $\eta_{\kappa,0}\in W^{2,2}(\omega)$, for  $\Vert\eta_{\kappa,0}\Vert_{L^\infty(\omega)}<\alpha<L$, we can use  \cite[Remark 2.24 \& Proposition 2.19]{lengeler2014weak} to conclude that there exists a bounded linear operator that takes $\eta_*$ as input and output a solenoidal vector field $\bm{\psi}_{\kappa,0}\in L^2_{\divx}(\Omega\cup S_{\alpha})$ that satisfies $\bm{\psi}_{\kappa,0}\circ\bm{\varphi}_{\eta_{\kappa,0}}=\eta_*\bn$ on $\omega$. More precisely,
\begin{align*}
\bm{\psi}_{\kappa,0}
:=
\exp\bigg(-\int_{\eta_0}^{\eta_{\kappa,0}} \divx\bn \, (\bm{\varphi}+\tau\bn )\dd\tau \bigg)\eta_*\bn.
\end{align*}
Using this vector field, we build $\bu_{\kappa,0}$ by defining
\begin{align*} 
\bu_{\kappa,0}:=
 \left\{
  \begin{array}{lr}
    \bu_0 & : \text{ in } \Omega_{\eta_0},\\
    \bm{\psi}_{\kappa,0}
     & : \text{ in }\Omega_{\eta_{\kappa,0}}\setminus \Omega_{\eta_0}.
  \end{array}
\right.
\end{align*}
which lead to  $\bu_{\kappa,0}\in L^2_{\divx}(\Omega_{\eta_{\kappa,0}})$  and
\begin{align*} 
  \mathbb{I}_{\Omega_{\eta_{\kappa,0}}} \bu_{\kappa,0} \rightarrow 
   \mathbb{I}_{\Omega_{\eta_0}}  \bu_0
&\qquad\text{in}\qquad L^2(\Omega\cup S_\alpha).
\end{align*}
To now ensure that $\bu_{\kappa,0}\circ\bm{\varphi}_{\eta_{\kappa,0}} =\eta_{\kappa,*}\bn$ hold on $\omega$, the function $\eta_{\kappa,*}$   is now defined as
\begin{align*}
\eta_{\kappa,*}:=
\exp\bigg(-\int_{\eta_0}^{\eta_{\kappa,0}} \divx\bn \, (\bm{\varphi}+\tau\bn )\dd\tau \bigg)\eta_*
\end{align*}
from which the convergence
\begin{align*}
\eta_{\kappa,*}\rightarrow   \eta_*
&\qquad\text{in}\qquad  L^2( \omega)
\end{align*}
also hold.

The director field can be treated in a straightfoward manner as the construction of $\eta_{\kappa,0}$ by convolution with the standard mollifier leading to
\begin{align*} 
  \mathbb{I}_{\Omega_{\eta_{\kappa,0}}} \bd_{\kappa,0} \rightarrow 
   \mathbb{I}_{\Omega_{\eta_0}}  \bd_0
&\qquad\text{in}\qquad W^{1,2}(\Omega\cup S_\alpha).
\end{align*}  
Having constructed the desired dataset, let now make precise, the notion of a weak solution $(\eta_\kappa,\bu_\kappa,\bd_\kappa)$ in this linearized setting.

\begin{definition}[Weak solution] \label{def:weakSolutionKappa}
Let $(\eta_{\kappa,0}, \eta_{\kappa,*}, \bu_{\kappa,0},\bd_{\kappa,0} )$ be a dataset such that
\begin{equation}
\begin{aligned}
\label{datasetKappaLin}
&
\eta_{\kappa,0} \in W^{2,2}(\omega) \text{ with } \Vert \eta_{\kappa,0}\Vert_{L^\infty(\omega)}<L , 
\\
& 
\eta_{\kappa,*} \in L^2(\omega),
\quad 
\bu_{\kappa,0}\in L^2_{\mathrm{\Div}}(\Omega_{\eta_{\kappa,0}})   
\\&
\text{and }\bu_{\kappa,0}\circ\bm{\varphi}_{\eta_{\kappa,0}} =\eta_{\kappa,*}\bn\text{ on $\omega$}, 
\\&
\bd_{\kappa,0} \in W^{1,2}(\Omega_{\eta_{\kappa,0}}),
\quad
\vert \bd_{\kappa,0}\vert\leq1 \text{  in } \Omega_{\eta_{\kappa,0}}.
\end{aligned}
\end{equation} 
We call  
$(\eta_\kappa,\bu_\kappa,\bd_\kappa)$
a weak solution to the system \eqref{contEqKappaLin}--\eqref{directorEqKappaLin} with data $(\eta_{\kappa,0}, \eta_{\kappa,*}, \bu_{\kappa,0},\bd_{\kappa,0} )$ provided that the following holds:
\begin{itemize}
\item[(a)] The structure displacement $\eta_\kappa$ satisfies
\begin{align*}
\eta_\kappa \in W^{1,\infty} \big(I; L^2(\omega) \big)\cap W^{1,2} \big(I; W^{1,2}(\omega) \big)\cap  L^\infty \big(I; W^{2,2}(\omega) \big) 
\quad \text{with} \quad \Vert \eta_\kappa \Vert_{L^\infty(I \times \omega)} <L,
\end{align*}
as well as $\eta_\kappa(0,\by)=\eta_{\kappa,0}(\by)$ and $\partial_t\eta_\kappa(0,\by)=\eta_{\kappa,*}(\by)$.
\item[(b)] The velocity field $\bu$ satisfies
\begin{align*}
 \bu_\kappa \in L^\infty \big(I; L^2(\Omega_{\zeta_\kappa}) \big)\cap  L^2 \big(I; W^{1,2}_{\Div}(\Omega_{\zeta_\kappa}) \big) \quad \text{with} \quad 
\bu_\kappa \circ\bm{\varphi}_{\zeta_\kappa} = \partial_t\eta_\kappa\bn\quad\text{on}\quad I\times\omega,
\end{align*}
as well as $\bu_\kappa(0,\bx)=\bu_{\kappa,0}(\bx)$. 
\item[(c)] The director field $\bd$ satisfies
\begin{align*}
 \bd_\kappa \in L^\infty \big(I; W^{1,2}(\Omega_{\zeta_\kappa}) \big)\cap  L^2 \big(I; W^{2,2} (\Omega_{\zeta_\kappa}) \big)  
 \quad \text{with} \quad 
\vert \bd_\kappa\vert\leq1 \quad\text{in}\quad I\times\Omega_{\zeta_\kappa},
\end{align*}
as well as $\bd_\kappa(0,\bx)=\bd_{\kappa,0}(\bx)$.
\item[(d)] For all  $(\phi, {\bfphi}, \bm{\psi}) \in C^\infty(\overline{I}\times\omega) \otimes C^\infty(\overline{I} ;C^\infty_{\Div}(\R^2)) \otimes C^\infty(\overline{I}\times\R^2) $ with $\phi(T,\cdot)=0$, ${\bfphi}(T,\cdot)=0$ and $\bfphi\circ\bfvarphi_{\zeta_\kappa} =\phi {\bn}$ on $I\times\omega$, we have
\begin{equation}
\begin{aligned}
\label{galerkinweak2}
\int_I\frac{\dd}{\dt}\bigg(\int_{\omega}&  \partial_t \eta_\kappa \, \phi\dy
+
\int_{\Omega_{\zeta_\kappa}}
\bu_\kappa  \cdot \bm{\phi}\dx
+
\int_{\Omega_{\zeta_\kappa}}
\bd_\kappa  \cdot \bm{\psi}  \dx
\bigg)\dt
\\
&=
\int_I
\int_{\omega} \bigg(
\partial_t\eta_\kappa\partial_t\phi
-
\partial_t\naby\eta_\kappa\naby\phi -
\naby^2\phi \,\naby^2\eta_\kappa\bigg)\dy\dt
\\&+
\int_I
\int_{\omega} \bigg(\tfrac{1}{2}\bn_{\zeta_\kappa}\circ\bm{\varphi}_{\zeta_\kappa}\cdot \bn^\top \phi\,  \partial_t\zeta_\kappa \,\partial_t\eta_\kappa  \, 
 |\partial_{y } \bm{\varphi}_{\zeta_\kappa } |
  \bigg)\dy\dt
\\
&+
\int_I
 \int_{\Omega_{\zeta_\kappa}}\Big(  \bu_\kappa\cdot\partial_t \bm{\phi} 
-
\tfrac{1}{2}((\bv_\kappa\cdot\nabx)\bu_\kappa)\cdot \bm{\phi} \Big) \dx\dt
\\
&+
\int_I
 \int_{\Omega_{\zeta_\kappa}}\Big(  
\tfrac{1}{2}((\bv_\kappa\cdot\nabx)  \bm{\phi}) \cdot \bu_\kappa
-\nabx \bu_\kappa:\nabx   \bm{\phi}
+ \nabx \bd_\kappa\odot \nabx\bd_\kappa:\nabx\bm{\phi} \Big) \dx\dt
\\
&+
\int_I
 \int_{\Omega_{\zeta_\kappa}}\Big(  \bd_\kappa\cdot\partial_t\bm{\psi}  
+
\bu_\kappa\otimes \bd_\kappa:\nabx \bm{\psi}
+
\Delx \bd_\kappa\cdot \bm{\psi} - \tfrac{1}{\epsilon^2}( \vert \bd_\kappa\vert^2 - 1)\bd_\kappa\cdot \bm{\psi} \Big) \dx\dt
\end{aligned}
\end{equation}
\item[(e)]  
the estimate 
\begin{equation}
\begin{aligned}
\label{energyEstD1aX} 
 \frac{1}{2}\int_\omega& \big( \vert\partial_t \eta_\kappa \vert^2+ \vert\naby^2\eta_\kappa \vert^2\big)(t) \dy
+ 
 \frac{1}{2}
\int_{\Omega_{\zeta_\kappa}} \Big(\vert   \bu_\kappa  \vert^2
+
\vert\nabx\bd_\kappa  \vert^2
+
\tfrac{1}{2\epsilon^2}(\vert\bd_\kappa \vert^2 - 1)^2\Big)(t)
 \dx
\\& 
+
\int_0^t
\int_\omega \vert\partial_{t'}\naby \eta_\kappa\vert^2\dy \dt'
+
\int_0^t
 \int_{\Omega_{\zeta_\kappa}} \Big( \vert \nabx \bu_\kappa \vert^2
 +
 \big\vert\Delx\bd_\kappa - \tfrac{1}{\epsilon^2}(\vert \bd_\kappa \vert^2 - 1)\bd_\kappa\big\vert^2 
 \Big)
  \dx\dt'
\\&
\leq
 \frac{1}{2}
\int_\omega \big( \vert\eta_{\kappa,*}\vert^2+ \vert\naby^2\eta_{\kappa,0}\vert^2\big) \dy 
+ 
 \frac{1}{2}
\int_{\Omega_{\eta_0}} \Big(\vert   \bu_{\kappa,0}  \vert^2
+
\vert\nabx\bd_{\kappa,0}  \vert^2
+
\tfrac{1}{2\epsilon^2}(\vert\bd_{\kappa,0} \vert^2 - 1)^2\Big) 
 \dx 
\end{aligned}
\end{equation}
holds for all $t\in I$.
\end{itemize}
\end{definition}
Our main result in this section is now given as follows.  
\begin{theorem}
\label{thm:linReg}
Let $\kappa>0$ be fixed.
For a dataset $( \eta_{\kappa,0}, \eta_{\kappa,*}, \bu_{\kappa,0},\bd_{\kappa,0} )$ satisfying \eqref{datasetKappaLin},  a weak solution $(\eta_\kappa,\bu_\kappa,\bd_\kappa)$ of \eqref{contEqKappaLin}--\eqref{directorEqKappaLin} exists.
\end{theorem}
 \begin{proof} 
For a fixed $\kappa>0$, the solution $(\eta_\kappa,\bu_\kappa,\bd_\kappa)$ can de derived a limit $N\rightarrow\infty$ of a finite-dimensional triple $(\eta^N_\kappa,\bu^N_\kappa,\bd^N_\kappa)$.
To justify the existence of the triple $(\eta^N_\kappa,\bu^N_\kappa,\bd^N_\kappa)$,  we consider the basis $(\overline{\bm{X}}_n)_{n\in \mathbb{N}}$ and $(\overline{Y}_n)_{n\in \mathbb{N}}$ of $W^{2,2}_{0,\divx}(\Omega)$ and $W^{2,2}(\omega)$, respectively. Then by \cite[Theorem A.3]{lengeler2014weak}, there exists solenoidal vector fields $(\overline{\bm{Y}}_n)_{n\in \mathbb{N}}$ that solve a Stokes system on the fixed reference domain with boundary data $(\bn\overline{Y}_n)_{n\in \mathbb{N}}$. Now, for all $t\in \overline{I}$, we obtain from $(\bm{\overline{X}}_n)_{n\in \mathbb{N}}$, the following basis
\begin{align*}
 \bm{X}_n(t,\cdot):=\mathcal{J}_{\zeta_\kappa(t)}\overline{\bm{X}}_n
\end{align*}
for $W^{2,2}_{0,\divx}(\Omega_{\zeta_\kappa})$ where $\mathcal{J}_{\zeta}$, defined by
\begin{align}
\label{piolaTransform}
\mathcal{J}_{\zeta}\bv
=
\big(\nabx  \bm{\Psi}_{\zeta}(\mathrm{det}\nabx  \bm{\Psi}_{\zeta})^{-1}
\bv
\big)\circ \bm{\Psi}_{\zeta}^{-1}
\end{align}
is the Piola transform of the vector field $\bv$ with respect to the given mapping $\zeta:\omega\rightarrow\mathbb{R}$. The Piola transform is invertible with inverse
\begin{align}
\label{piolaTransformInverse}
\mathcal{J}_{\zeta}^{-1}\bv
=
\big((\nabx  \bm{\Psi}_{\zeta})^{-1}(\mathrm{det}\nabx  \bm{\Psi}_{\zeta}) 
\bv
\big)\circ \bm{\Psi}_{\zeta} 
.
\end{align} 
In order to ensure that the basis for the fluid system matches with the basis for the shell at the fluid-shell interface, we consider the following transform of the  solenoidal vector fields $(\overline{\bm{Y}}_n)_{n\in \mathbb{N}}$ 
\begin{align*}
\bm{Y}_n(t,\cdot):=\mathcal{J}_{\zeta_\kappa(t)}\overline{\bm{Y}}_n.
\end{align*} 
If we now set
\begin{align*}
Y_n(t,\cdot):=|\naby\bm{\varphi}_{\zeta_\kappa(t)}|^{-1}\overline{Y}_n
\end{align*}
we obtain the interface condition
\begin{align*}
Y_n(t,\cdot)\bn=\bm{Y}_n(t,\cdot)\circ\bm{\varphi}_{\zeta_\kappa(t)}.
\end{align*}
Consequently,  we obtain on the moving domain, the extended pair of basis $(\bm{\psi}_n )_{n\in \mathbb{N}}$ and $(\psi_n )_{n\in \mathbb{N}}$ for $W^{2,2}_{0,\divx}(\Omega_{\zeta_\kappa})$ and $W^{2,2}(\omega)$, respectively, where
\begin{equation}
\label{basis}
\bm{\psi}_n  = \left\{
  \begin{array}{lr}
    \bm{X}_n & : n \text{ even},\\
    \bm{Y}_n & : n \text{ odd}.
  \end{array}
\right.
\qquad\text{and}\qquad
\psi_n\bn:=\bm{\psi}_n\circ \bm{\varphi}_{\zeta_\kappa}.
\end{equation}
Having constructed the basis \eqref{basis}, we can employ Picard--Lindel\"of's theorem to find functions $\alpha^N_n\in C^1(\overline{I})$, $n,N\in \mathbb{N}$
satisfying $\sum_{n=1}^N|\alpha^N_n(t)|^2\leq R^2$ for all $t\in \overline{I}$ and for some $R>0$ such that $\bu^N_\kappa:=\sum_{n=1}^N\alpha^N_n\bm{\psi}_n$ and $\eta^N_\kappa=\sum_{n=1}^N\int_0^t\alpha^N_n\psi_n\ds+\eta_{\kappa,0}$ solves\footnote{The dependence of $\bu^N_\kappa$ and $\eta^N_\kappa$ on $\kappa$ follows from the implicit dependence of $\bm{\psi}_n$ and $\psi_n$ on $\kappa$.}
\begin{equation}
\begin{aligned}
\label{galerkinweak1} 
\int_{\omega}&  \partial_t^2 \eta^N_\kappa \, \psi_j\dy
+
\int_{\Omega_{\zeta_\kappa}}
\partial_t\bu^N_\kappa  \cdot  \bm{\psi}_j \dx 
 =-
\int_{\omega} \Big( 
\partial_t\naby\eta^N_\kappa\naby\psi_j 
+
\naby^2\eta^N_\kappa \naby^2 \psi_j \Big)\dy
\\&-
\frac{1}{2}
\int_{\omega}  \bn_{\zeta_\kappa}\circ\bm{\varphi}_{\zeta_\kappa}\cdot \bn^\top \psi_j \,  \partial_t\zeta_\kappa \,\partial_t\eta^N_\kappa  \, 
  |\partial_{y } \bm{\varphi}_{\zeta_\kappa } | \dy
 -
\frac{1}{2}
 \int_{\Omega_{\zeta_\kappa}} 
((\bv_\kappa\cdot\nabx)\bu^N_\kappa)\cdot  \bm{\psi}_j  \dx
\\
&+
 \int_{\Omega_{\zeta_\kappa}}\Big(  
\tfrac{1}{2}((\bv_\kappa\cdot\nabx) \bm{\psi}_j) \cdot \bu^N_\kappa
-\nabx \bu^N_\kappa:\nabx  \bm{\psi}_j  
+  \nabx\bd^N_\kappa\odot \nabx\bd^N_\kappa:\nabx \bm{\psi}_j \Big) \dx
\end{aligned}
\end{equation} 
for all $1\leq j\leq N$ subject to an initial condition $\alpha_n^N(0)$ which is such that
\begin{align*}
&\partial_t\eta_\kappa^N(0,\cdot)\rightarrow \eta_{\kappa,*}(\cdot)&\text{ in } L^2(\omega),
\\
&\bu_\kappa^N(0,\cdot) \rightarrow \bu_{\kappa,0}(\cdot)
&\text{ in } L^2(\Omega_{\zeta_\kappa(0)}),
\end{align*}
and
\begin{align*}
&\Vert \partial_t\eta_\kappa^N(0)\Vert_{L^2(\omega)}\leq \Vert \eta_{\kappa,*}\Vert_{L^2(\omega)}, 
\\&
\Vert \bu_\kappa^N(0)\Vert_{L^2(\Omega_{\zeta_\kappa(0)})} \leq \Vert\bu_{\kappa,0}\Vert_{L^2(\Omega_{\zeta_\kappa(0)})}.
\end{align*}
In particular, we can choose
\begin{align*}
\alpha^N_n(0)=\int_{\Omega_{\zeta_\kappa(0)}}\bu_{\kappa ,0}\cdot\bm{\psi}_n\dx,\qquad n=1,\ldots,N
\end{align*}
and set
\begin{align*}
\partial_t\eta^N_\kappa(0)&=\sum_{n=1}^N\alpha^N_n(0)\psi_n,
\qquad
\bu^N_\kappa (0)
=
\sum_{n=1}^N\alpha^N_n(0) \bm{\psi}_n\dx. 
\end{align*}
In \eqref{galerkinweak1}, $\bd^N_\kappa$ is determined by the equation
\begin{align}
\label{approxDAlone}
\partial_t \bd^N_\kappa + ( \bu^N_\kappa\cdot \nabx) \bd^N_\kappa
=
   \Delx \bd^N_\kappa -\tfrac{1}{\epsilon^2}(\vert \bd^N_\kappa\vert^2 - 1)\bd^N_\kappa 
&\qquad \text{in }I \times  \Omega_{\zeta_\kappa}
\end{align}
subject to the boundary  and  initial conditions 
\begin{align*}
&\partial_{\bn_{\zeta_\kappa}}\bd^N_\kappa=0
&\quad \text{on }I \times \partial\Omega_{\zeta_\kappa},
\\
&\bd^N_\kappa(0, \bx)=\bd_{\kappa,0}(\bx)  
&\quad \text{in }  \Omega_{\zeta_\kappa(0    )},
\end{align*}
respectively.  
Indeed, for a fixed $N\in\mathbb{N}$, suppose that $(\overline{\eta^N_\kappa},\overline{\bu^N_\kappa})$ satisfying
\begin{align*}
&\overline{\eta^N_\kappa}\in W^{1,\infty} \big(I; L^2(\omega) \big)\cap W^{1,2} \big(I; W^{1,2}(\omega) \big)\cap  L^\infty \big(I; W^{2,2}(\omega) \big),
\\
&\overline{\bu^N_\kappa}\in L^\infty \big(I; L^2(\Omega_{\zeta_\kappa}) \big)\cap  L^2 \big(I; W^{1,2}_{\Div}(\Omega_{\zeta_\kappa}) \big)
\\&
\overline{\bu^N_\kappa}:=\sum_{n=1}^N\overline{\alpha^N_n}\bm{\psi}_n\quad \text{ and }\quad \overline{\eta^N_\kappa}=\sum_{n=1}^N\int_0^t\overline{\alpha^N_n}\psi_n\ds+\eta_{\kappa,0}
\end{align*}
is given.  A  solution $\bd^N_\kappa$ 
satisfying
\begin{equation}
\begin{aligned}
\label{approxDAloneX}
\frac{\dd}{\dt}\int_{\Omega_{\zeta_\kappa}} \bd^N_\kappa\cdot \bm{\phi}_i\dx
&=
 \int_{\Omega_{\zeta_\kappa}}\big( \bd^N_\kappa\cdot \partial_t\bm{\phi}_i
+
\overline{\bu^N_\kappa}\otimes \bd^N_\kappa:\nabx \bm{\phi}_i
\big) \dx
\\
&+
 \int_{\Omega_{\zeta_\kappa}}\big( 
\Delx \bd^N_\kappa  - \tfrac{1}{\epsilon^2}(\vert \bd^N_\kappa\vert^2 - 1)\bd^N_\kappa
\big) \cdot  \bm{\phi}_i\dx
\end{aligned}
\end{equation}
for all $1\leq i\leq M$ 
can be constructed as a limit $M\rightarrow\infty$ of, yet again, a finite-dimensional Galerkin approximation $\bd^{N,M}_\kappa$ where  $(\bm{\phi}_i)_{i\in \mathbb{N}}$ is a basis of $W^{1,2}_{0}(\Omega_{\zeta_\kappa})$. Note that unlike the fluid-structure subsystem, here, we do not require the basis to be incompressible nor do we require them to satisfy an interface condition.
\\
By  testing \eqref{approxDAlone} with $\Delx\bd^N_\kappa-\tfrac{1}{\epsilon^2}(\vert \bd^N_\kappa\vert^2 - 1)\bd^N_\kappa$, we find that similar to \eqref{energyEstB},  this solution satisfies 
\begin{align*} 
\int_{\Omega_{\zeta_\kappa}}   
\vert\nabx\bd^N_\kappa(t) \vert^2 \dx
&+
\frac{1}{\epsilon^2} 
\int_{\Omega_{\zeta_\kappa}} 
 (\vert \bd^N_\kappa(t) \vert^2 - 1)^2\dx
 +
 \int_0^t\int_{\Omega_{\zeta_\kappa}}  \big\vert\Delx\bd^N_\kappa -  \tfrac{1}{\epsilon^2}(\vert \bd^N_\kappa \vert^2 - 1)\bd^N_\kappa\big\vert^2  \dx\dt'
 \\&
\leq
 \mathcal{E}(\bd_{\kappa,0} )
 -
  \int_0^t\int_{\Omega_{\zeta_\kappa}}  \nabx \bd^N_\kappa\odot \nabx\bd^N_\kappa:\nabx \overline{\bu^N_\kappa}  \dx\dt'
\end{align*}
for all $t\in I$ where
\begin{align*}
\mathcal{E}(\bd_{\kappa,0} ):=&
 \int_{\Omega_{\zeta_\kappa(0)}}   
\vert\nabx\bd_{\kappa,0} \vert^2 \dx
+
\int_{\Omega_{\zeta_\kappa(0)}} \tfrac{1}{\epsilon^2}
 (\vert \bd_{\kappa,0} \vert^2 - 1)^2\dx 
\end{align*}
and finally, by using the equivalence of norms on finite-dimensional spaces, we obtain from the definition of $\overline{\bu^N_\kappa}$,
\begin{align*}
\bigg\vert
  \int_0^t\int_{\Omega_{\zeta_\kappa}}  \nabx \bd^N_\kappa\odot \nabx\bd^N_\kappa:\nabx \overline{\bu^N_\kappa}  \dx\dt'
  \bigg\vert
&\lesssim
\int_0^t
\Vert \nabx \overline{\bu^N_\kappa}
\Vert_{L^\infty(\Omega_{\zeta_\kappa})}
\Vert\nabx \bd^N_\kappa
\Vert_{L^2(\Omega_{\zeta_\kappa})}^2
\dt'
\\&
\lesssim R^2  \int_0^t\sup_{1\leq n\leq N}
\Vert \nabx\bm{\psi}_n
\Vert_{L^\infty(\Omega_{\zeta_\kappa})}
\Vert\nabx \bd^N_\kappa
\Vert_{L^2(\Omega_{\zeta_\kappa})}^2\dt'
\\&
\lesssim\int_0^t
 \Vert\nabx \bd^N_\kappa
\Vert_{L^2(\Omega_{\zeta_\kappa})}^2\dt'.
\end{align*} 
It follows from Gr\"onwall's lemma that for $c=c(R)$, we have that  the inequality
\begin{equation}
\begin{aligned}
\label{finiteDirector} 
\sup_{t\in I}
\int_{\Omega_{\zeta_\kappa}}   
\vert\nabx\bd^N_\kappa \vert^2 \dx
&+
\frac{1}{\epsilon^2}\sup_{t\in I}
\int_{\Omega_{\zeta_\kappa}} 
 ( \vert \bd^N_\kappa \vert^2 - 1)^2\dx
 +
 \int_{I}\int_{\Omega_{\zeta_\kappa}}  \big\vert\Delx\bd^N_\kappa -  \tfrac{1}{\epsilon^2}( \vert \bd^N_\kappa \vert^2 - 1)\bd^N_\kappa\big\vert^2  \dx\dt
 \\&
 \leq c \,\mathcal{E}(\bd_{\kappa,0})\big(1+T\exp( c T) \big) 
\end{aligned}
\end{equation}
holds.  
Now  note that by the relation $\frac{x^2}{2}-y^2\leq (x-y)^2$ and Agmon's inequality, the estimate 
\begin{align*}
\frac{1}{2}\int_{I}\int_{\Omega_{\zeta_\kappa}}\vert \Delx \bd^N_\kappa\vert^2\dx\dt
&\leq
\int_{I}\int_{\Omega_{\zeta_\kappa}}  \tfrac{1}{\epsilon^4}(\vert \bd^N_\kappa \vert^2 - 1)^2\vert\bd^N_\kappa\vert^2\dx\dt
\\
&\quad
+
\int_{I}
\int_{\Omega_{\zeta_\kappa}}\vert\Delx\bd^N_\kappa -  \tfrac{1}{\epsilon^2}(\vert \bd^N_\kappa \vert^2 - 1)\bd^N_\kappa\big\vert^2\dx\dt
\\
&\lesssim 
\int_{I}\Vert \Delx\bd^N_\kappa\Vert_{L^2(\Omega_{\zeta_\kappa})}\dt
\sup_{t\in I}
\Vert \nabx\bd^N_\kappa\Vert_{L^2(\Omega_{\zeta_\kappa})}
\sup_{t\in I}
\int_{\Omega_{\zeta_\kappa}}  \tfrac{1}{\epsilon^4}(\vert \bd^N_\kappa \vert^2 - 1)^2\dx
\\&+\quad
\int_{I}
\int_{\Omega_{\zeta_\kappa}}\vert\Delx\bd^N_\kappa - \tfrac{1}{\epsilon^2}(\vert \bd^N_\kappa \vert^2 - 1)\bd^N_\kappa\big\vert^2\dx\dt
\\
&\leq
\frac{1}{4}
\int_{I}\Vert \Delx\bd^N_\kappa\Vert_{L^2(\Omega_{\zeta_\kappa})}^2\dt
+c\sup_{t\in I}
\Vert \nabx\bd^N_\kappa\Vert_{L^2(\Omega_{\zeta_\kappa})}^2
\bigg(\sup_{t\in I}
\int_{\Omega_{\zeta_\kappa}}  \tfrac{1}{\epsilon^4}(\vert \bd^N_\kappa \vert^2 - 1)^2\dx\bigg)^2
\\&+\quad
c\int_{I}
\int_{\Omega_{\zeta_\kappa}}\vert\Delx\bd^N_\kappa -   \tfrac{1}{\epsilon^2}(\vert \bd^N_\kappa \vert^2 - 1)\bd^N_\kappa\big\vert^2\dx\dt
\end{align*} 
holds. Therefore, it follows from \eqref{finiteDirector} and Young's inequality that
\begin{align}
\label{finiteDirector1}
\int_{I}\int_{\Omega_{\zeta_\kappa}}\vert \Delx \bd^N_\kappa\vert^2\dx\dt
\lesssim
c(R )\mathcal{E}(\bd_{\kappa,0} )^4\big(1+T\exp( c T) \big) ^4.
\end{align} 
We have thus constructed a solution $\bd^N_\kappa$ of \eqref{approxDAlone} that satisfies
\begin{align}
\label{finiteDirector2}
\bd^N_\kappa\in L^\infty\big(I; W^{1,2}(\Omega_{\zeta_\kappa}\big)
\cap 
L^2\big(I; W^{2,2}(\Omega_{\zeta_\kappa})\big).
\end{align}
Now for the constructed solution $\bd^N_\kappa$ of \eqref{approxDAlone} (for fixed $\kappa>0$) satisfying \eqref{finiteDirector2}, let us now return to the equation \eqref{galerkinweak1}. Since $\bu^N_\kappa:=\sum_{n=1}^N\alpha^N_n\bm{\psi}_n$ and $\eta^N_\kappa=\sum_{n=1}^N\int_0^t\alpha^N_n\psi_n\ds+ \eta_{\kappa,0}$, we can rewrite the following terms in \eqref{galerkinweak1} as follows:
\begin{align*}
 \int_{\Omega_{\zeta_\kappa}}
\partial_t\bu^N_\kappa  \cdot  \bm{\psi}_j \dx
&=
\frac{\dd \alpha_n^N }{\dt}\int_{\Omega_{\zeta_\kappa}}
\bm{\psi}_n  \cdot  \bm{\psi}_j \dx
+ 
 \alpha_n^N
 \int_{\Omega_{\zeta_\kappa}}
\partial_t\bm{\psi}_n  \cdot  \bm{\psi}_j \dx,
\\ 
\int_{\omega}  \partial_t^2 \eta^N_\kappa \, \psi_j\dy
&=
\frac{\dd \alpha_n^N }{\dt}
\int_{\omega} 
   \psi_n \, \psi_j\dy 
+
\alpha_n^N  
\int_{\omega} 
 \partial_t\psi_n \, \psi_j\dy , 
\\
-
\int_{\omega} 
\partial_t\naby\eta^N_\kappa\naby\psi_j\dy
&=
-
\int_\omega
\alpha^N_n \naby\psi_n\naby\psi_j
\dy,
\\
-\int_{\omega}
\naby^2 \psi_j \naby^2\eta^N_\kappa \dy
 &=-
 \int_\omega
\naby^2 \psi_j\int_0^t\alpha^N_n \naby^2\psi_n \ds
  - 
  \int_\omega
\naby^2 \psi_j \naby^2\eta_{\kappa,0}  \dy.
\end{align*} 
If we sum up all the identities, it follows that
\eqref{galerkinweak1} is equivalent to the following finite-dimensional system of integro-differential equations
\begin{equation}
\begin{aligned}
\label{galerkinweak3}
a_{nj}(t)\frac{\dd}{\dt}  \alpha^N_n(t)  
&= 
 \alpha^N_n(t)(b_{nj}(t))   
  + 
   \int_0^t\alpha^N_n(s)(c_{nj}(t,s)) \dd s  
 +
   (d_j(t)) 
   +
   (e_j(t)) 
\end{aligned}
\end{equation}
for all $1\leq j\leq N$ where
\begin{align*}
(a_{nj}(t))
&:=
\int_{\Omega_{\zeta_\kappa}}\bm{\psi}_n \cdot \bm{\psi}_j \dx
+
\int_{\omega}   \psi_n \, \psi_j\dy,
\\
(b_{nj}(t))
&:=
 \int_{\Omega_{\zeta_\kappa}}
 \Big( -\partial_t\bm{\psi}_n\cdot \bm{\psi}_j 
-
\tfrac{1}{2}(\bv_\kappa\cdot\nabx)\bm{\psi}_n\cdot\bm{\psi}_j
+
\tfrac{1}{2}(\bv_\kappa\cdot\nabx) \bm{\psi}_j \cdot \bm{\psi}_n -
  \nabx \bm{\psi}_n:\nabx \bm{\psi}_j \Big) \dx
 \\&
 -
 \int_{\omega} \Big(
 \tfrac{1}{2} \bn_{\zeta_\kappa}\circ\bm{\varphi}_{\zeta_\kappa}\cdot \bn^\top \psi_j \,  \partial_t\zeta_\kappa\, \psi_n \, 
 |\partial_{y } \bm{\varphi}_{\zeta_\kappa } |
 +
 \partial_t\psi_n\psi_j
+
 \naby\psi_n\cdot\naby\psi_j 
 \Big)\dy, 
\\
(c_{nj}(t,s))
&:=
-
\int_{\omega}  
\naby^2 \psi_j(t) \naby^2 \psi_n(s) \dy,
\\
(d_{j}(t))
&:= -
   \int_{\omega}  
\naby^2 \eta_{\kappa,0}\, \naby^2 \psi_j \dy, 
 \\
(e_{j}(t))
&:=  
 \int_{\Omega_{\zeta_\kappa}}  
  \nabx \bd^N_\kappa\odot\nabx \bd^N_\kappa:\nabx \bm{\psi}_j \dx.
\end{align*}  
Now, let us note that it follows from \eqref{finiteDirector} that for any smooth $\bm{\psi}_j$
\begin{equation}
\begin{aligned}
\label{eijEst}
\sup_{t\in I}\sum_{j=1}^N\big\vert (e_{j}(t)) \big\vert
&\lesssim 
\sup_{t\in I}
\int_{\Omega_{\zeta_\kappa}}   
\vert\nabx\bd^N_\kappa \vert^2 \dx
\sup_{1\leq j\leq N}
\Vert\nabx\bm{\psi}_j
\Vert_{L^\infty(\Omega_{\zeta_\kappa})} 
\\&
 \leq 
 c( R)\mathcal{E}(\bd_{\kappa,0}  )\big(1+T\exp( c T) \big) 
\end{aligned}
\end{equation}
and thus, the coefficient $(e_j(t))$ is continuous in time and also trivially Lipschitz continuous with respect to $\alpha^N_n$ since the latter is constant with respect to $\alpha^N_n$. Similarly, by using the smoothness of $\zeta_\kappa, \bv_\kappa$, we obtain continuity in time and Lipschitz continuity in $\alpha^N_n$  of the other coefficients. Furthermore, since the coefficient matrix $(a_{nj}(t))$ is symmetric and positive definite, it is invertible. Consequently, by Picard--Lindel\"of theorem  a unique local (with time $T_1\in(0, T]$, $I_1:=(0,T_1)$ depending on  $  \zeta_\kappa ,\bd_{\kappa,0}  $ and $\bv_\kappa$)  solution $\alpha^N_n\in C^1(\overline{I}_1)$ of \eqref{galerkinweak3} exists for $1\leq n\leq N$. 
  
If we now multiply \eqref{galerkinweak1} by $\alpha_j^N$ so that
$\bu^N_\kappa:=\sum_{n=1}^N\alpha^N_j\bm{\psi}_j$ and $\eta^N_\kappa:=\sum_{n=1}^N\int_0^t\alpha^N_j\psi_j\ds+ \eta_{\kappa,0}$, then just as in \eqref{energyEstC} and \eqref{energyEstA}, we obtain
\begin{equation}
\begin{aligned} 
\label{energyEstAx}
\frac{1}{2}
\frac{\dd}{\dt}
\int_{\Omega_{\zeta_\kappa}} \vert  \bu^N_\kappa \vert^2
 \dx
 &+
\frac{1}{2}
\frac{\dd}{\dt}
\int_\omega \big( \vert\partial_t \eta^N_\kappa\vert^2+ \vert\naby^2\eta^N_\kappa\vert^2\big) \dy
+
 \int_{\Omega_{\zeta_\kappa}}  \vert \nabx \bu^N_\kappa \vert^2
  \dx
  +
\int_\omega \vert\partial_t\naby \eta^N_\kappa\vert^2\dy 
\\&
=  
\int_{\Omega_{\zeta_\kappa}}   \nabx \bd^N_\kappa\odot \nabx\bd^N_\kappa:\nabx \bu^N_\kappa
\dx.
\end{aligned}
\end{equation}
By using Ladyzhenskaya's inequality,  the global bound
\begin{align*}
\int_I\int_{\Omega_{\zeta_\kappa}} &  \nabx \bd^N_\kappa\odot \nabx\bd^N_\kappa:\nabx \bu^N_\kappa
\dx\dt
\lesssim
\int_I\Vert \nabx \bd^N_\kappa\Vert_{L^4(\Omega_{\zeta_\kappa})}^2\Vert\nabx \bu^N_\kappa
\Vert_{L^2(\Omega_{\zeta_\kappa})}
 \dt
\\&\leq
\frac{1}{2}
\int_I\Vert\nabx \bu^N_\kappa
\Vert_{L^2(\Omega_{\zeta_\kappa})}^2
 \dt
 +c
 \int_I\Vert \Delx \bd^N_\kappa\Vert_{L^2(\Omega_{\zeta_\kappa})}^2\sup_{t\in I}\Vert\nabx \bd^N_\kappa
\Vert_{L^2(\Omega_{\zeta_\kappa})}^2
 \dt
\end{align*}
follow.
Thus, if we use \eqref{finiteDirector}-\eqref{finiteDirector1}, we obtain
\begin{equation}
\begin{aligned} 
\label{energyEstAxx} 
\sup_{t\in I}\int_{\Omega_{\zeta_\kappa}} \vert  \bu^N_\kappa \vert^2
 \dx
 &+ 
 \sup_{t\in I}
\int_\omega \big( \vert\partial_t \eta^N_\kappa\vert^2+ \vert\naby^2\eta^N_\kappa\vert^2\big) \dy
+
\int_{I}
 \int_{\Omega_{\zeta_\kappa}}  \vert \nabx \bu^N_\kappa \vert^2
  \dx\dt
  +
\int_{I}\int_\omega \vert\partial_t\naby \eta^N_\kappa\vert^2\dy \dt
\\&
\lesssim   
c(R)\mathcal{E}(\bd_{\kappa,0} )^4\big(1+T\exp( c T) \big) ^4
\\&\quad+
\int_{\Omega_{\zeta_\kappa(0)}} \vert  \bu_{\kappa,0} \vert^2
 \dx
 +
\int_\omega \big( \vert  \eta_{\kappa,*}\vert^2+ \vert\naby^2\eta_{\kappa,0}\vert^2\big) \dy .
\end{aligned}
\end{equation}
By using the finiteness of the right-hand side, we obtain a solution pair 
\begin{align*}
(\bu^N_\kappa,\eta^N_\kappa)\in X^{I}  
\end{align*}
where
\begin{align*}
 X^{I}:=&L^\infty \big(I; L^2(\Omega_{\zeta_\kappa}) \big)\cap  L^2 \big(I; W^{1,2}_{\Div}(\Omega_{\zeta_\kappa}) \big)
 \\&\otimes W^{1,\infty} \big(I; L^2(\omega) \big)\cap W^{1,2} \big(I; W^{1,2}(\omega) \big)\cap  L^\infty \big(I; W^{2,2}(\omega) \big).
\end{align*}
At this point, on the one hand, we have obtained a solution $\bd^N_\kappa$ of the director field for any given fluid-structure pair $(\overline{\eta^N_\kappa},\overline{\bu^N_\kappa})$, say. On the other hand, we have also constructed a fluid-structure pair $(\eta^N_\kappa,\bu^N_\kappa)$ given a director field $\bd^N_\kappa$. We can now use a fixed point argument to get a local solution $(\eta^N_\kappa,\bu^N_\kappa,\bd^N_\kappa)$ for the mutually coupled system.
To do this, for a time $T_*\in(0,T]$, $I_*:=(0,T_*)$ to be determined soon, we consider the solution map $\mathtt{T}=\mathtt{T}_1\circ \mathtt{T}_2: X^{I_*} \rightarrow  X^{I_*}$ where
\begin{align*}
\mathtt{T}(\overline{\eta^N_\kappa},\overline{\bu^N_\kappa})=(\eta^N_\kappa,\bu^N_\kappa), \quad \mathtt{T}_2(\overline{\eta^N_\kappa},\overline{\bu^N_\kappa})=\bd^N_\kappa, \quad \mathtt{T}_1(\bd^N_\kappa)=(\eta^N_\kappa,\bu^N_\kappa).
\end{align*}
Associated with this map is the set
\begin{align*}
B_R^{I_*}:=\Big\{ (\overline{\eta^N_\kappa},\overline{\bu^N_\kappa}) \in X^{\overline{I}_*}&\text{ with } \overline{\bu^N_\kappa}:=\sum_{n=1}^N\overline{\alpha}^N_n\bm{\psi}_n,\, \overline{\eta^N_\kappa}:=\sum_{n=1}^N\int_0^t\overline{\alpha}^N_n\psi_n\ds+\eta^N_{\kappa,0}, \, 
\\&
\text{ such that } \Vert (\overline{\eta^N_\kappa},\overline{\bu^N_\kappa}) \Vert_{X^{I_*}}^2  \leq R^2\text{ for }t\in \overline{I}_* \Big\}
\end{align*}
for $R>0$ large enough and for fixed $\kappa>0$. From  \eqref{energyEstAxx} and \eqref{finiteDirector}, it follows that for $T_*>0$ small enough, $\mathtt{T}:B_R^{I_*}  \subset X^{I_*}  \rightarrow  B_R^{I_*}$.  
Our goal now is to show the contraction property leading to the existence of a unique local solution for the fully coupled system. For this, we consider any two solution pair $(\eta^{N,i}_\kappa,\bu^{N,i}_\kappa)$, $i=1,2$ of the fluid-structure subsystem with datasets $( \eta^N_{\kappa}(0), \partial_t\eta^N_{\kappa}(0), \bu^N_{\kappa,0},\bd^{N,i}_\kappa)$, $i=1,2$, respectively.
Thus, the difference 
\begin{align*}
(\eta^{N,12}_\kappa\,,\, \bu^{N,12}_\kappa)
:=
(\eta^{N,1}_\kappa-\eta^{N,2}_\kappa\, ,\,\bu^{N,1}_\kappa - \bu^{N,2}_\kappa)
\end{align*}
solves 
\begin{align}
\partial_t^2\eta^{N,12}_\kappa - \partial_t\naby^2 \eta^{N,12}_\kappa + \naby^4\eta^{N,12}_\kappa
=
f^{N,12}(\zeta_\kappa,\bu^{N,12}_\kappa,\pi^{N,12}_\kappa,\bd^{N,12}_\kappa)
,
\end{align}
on $I_* \times \omega$
and
\begin{equation}
\begin{aligned}
\partial_t \bu^{N,12}_\kappa  + (\bv_\kappa\cdot \nabx)\bu^{N,12}_\kappa
=&
\divx\mathbb{T}(\bu^{N,12}_\kappa,\pi^{N,12}_\kappa,\bd^{N,12}_\kappa)
\end{aligned}
\end{equation}
in $I_* \times  \Omega_{\zeta_\kappa}$  with zero initial conditions and
where
\begin{align*}
&f^{N,12}(\zeta_\kappa,\bu^{N,12}_\kappa,\pi^{N,12}_\kappa,\bd^{N,12}_\kappa)
=
-(\mathbb{T}(\bu^{N,12}_\kappa, \pi^{N,12}_\kappa,\bd^{N,12}_\kappa )\bn_{\zeta_\kappa })\circ\bm{\varphi}_{\zeta_\kappa }\cdot\bn     
|\partial_{y } \bm{\varphi}_{\zeta_\kappa } |,
\end{align*}
and
\begin{align*}
\mathbb{T}(\bu^{N,12}_\kappa, \pi^{N,12}_\kappa,\bd^{N,12}_\kappa )=&(\nabx\bu^{N,12}_\kappa + \nabx(\bu^{N,12}_\kappa)^\top)-\pi^{N,12}_\kappa\mathbb{I}_{2\times2} 
-
\nabx \bd^{N,12}_\kappa\odot \nabx\bd^{N,1}_\kappa
\\&-
\nabx \bd^{N,2}_\kappa\odot \nabx\bd^{N,12}_\kappa.
\end{align*} 
Similar to \eqref{energyEstAx}, we obtain
\begin{equation}
\begin{aligned} 
\nonumber
\frac{1}{2}
\frac{\dd}{\dt}
\int_{\Omega_{\zeta_\kappa}} \vert  \bu^{N,12}_\kappa \vert^2
 \dx
 &+
\frac{1}{2}
\frac{\dd}{\dt}
\int_\omega \big( \vert\partial_t \eta^{N,12}_\kappa\vert^2+ \vert\naby^2\eta^{N,12}_\kappa\vert^2\big) \dy
\\&+
 \int_{\Omega_{\zeta_\kappa}}  \vert \nabx \bu^{N,12}_\kappa \vert^2
  \dx
  +
\int_\omega \vert\partial_t\naby \eta^{N,12}_\kappa\vert^2\dy 
\\
= & 
\int_{\Omega_{\zeta_\kappa}}   \nabx \bd^{N,1}_\kappa\odot \nabx\bd^{N,12}_\kappa:\nabx \bu^{N,12}_\kappa
\dx
\\&+
\int_{\Omega_{\zeta_\kappa}}   \nabx \bd^{N,2}_\kappa\odot \nabx\bd^{N,12}_\kappa:\nabx \bu^{N,12}_\kappa
\dx
\end{aligned}
\end{equation}
where using Ladyzhenskaya's inequality result in  
\begin{align*}
\int_{I_*}&
\int_{\Omega_{\zeta_\kappa}} (\nabx \bd^{N,1}_\kappa\odot \nabx\bd^{N,12}_\kappa+  \nabx \bd^{N,2}_\kappa\odot \nabx\bd^{N,12}_\kappa):\nabx \bu^{N,12}_\kappa
\dx\dt
\leq
\frac{1}{2}
\int_{I_*}
\Vert\nabx \bu^{N,12}_\kappa\Vert_{L^2(\Omega_{\zeta_\kappa})}^2
\dt
\\&+c
\sup_{t\in I_*}\Big(\Vert\nabx \bd^{N,1}_\kappa\Vert_{L^2(\Omega_{\zeta_\kappa})}^2
+
\Vert\nabx \bd^{N,2}_\kappa\Vert_{L^2(\Omega_{\zeta_\kappa})}^2\Big)
\int_{I_*}
\Vert\Delx \bd^{N,12}_\kappa\Vert_{L^2(\Omega_{\zeta_\kappa})}^2
\dt
\\&+c
\sup_{t\in I_*} \Vert\nabx \bd^{N,12}_\kappa\Vert_{L^2(\Omega_{\zeta_\kappa})}^2 
\int_{I_*}
\Big(
\Vert\Delx \bd^{N,1}_\kappa\Vert_{L^2(\Omega_{\zeta_\kappa})}^2
+
\Vert\Delx \bd^{N,2}_\kappa\Vert_{L^2(\Omega_{\zeta_\kappa})}^2
\Big)
\dt.
\end{align*}
Therefore, since
\begin{align*}
\bd^{N,1}_\kappa,\bd^{N,2}_\kappa\in L^\infty\big(I_*; W^{1,2}(\Omega_{\zeta_\kappa}\big)
\cap 
L^2\big(I_*; W^{2,2}(\Omega_{\zeta_\kappa})\big),
\end{align*}
it follows that
\begin{equation}
\begin{aligned} 
\label{energyEstAxxY} 
\sup_{t\in I_*}\int_{\Omega_{\zeta_\kappa}} \vert  \bu^{N,12}_\kappa  \vert^2
 \dx
 &+ 
 \sup_{t\in I_*}
\int_\omega \big( \vert\partial_t \eta^{N,12}_\kappa\vert^2+ \vert\naby^2\eta^{N,12}_\kappa\vert^2\big) \dy
\\&+
\int_{I_*}
 \int_{\Omega_{\zeta_\kappa}}  \vert \nabx \bu^{N,12}_\kappa \vert^2
  \dx\dt
  +
\int_{I_*}\int_\omega \vert\partial_t\naby \eta^{N,12}_\kappa\vert^2\dy \dt
\\
\lesssim& 
\sup_{t\in I_*} \Vert\nabx \bd^{N,12}_\kappa\Vert_{L^2(\Omega_{\zeta_\kappa})}^2
+
\int_{I_*}
\Vert\Delx \bd^{N,12}_\kappa\Vert_{L^2(\Omega_{\zeta_\kappa})}^2
\dt.
\end{aligned}
\end{equation}
To obtain a contraction estimate for $\bd^{N}_\kappa$,
let us consider any two solutions $(\bd^{N,i}_\kappa)$, $i=1,2$ of  \eqref{approxDAlone} with datasets $( \bd^N_{\kappa,0},  \overline{\bu^{N,i}_\kappa} )$, $i=1,2$, respectively. Then $\bd^{N,12}_\kappa $ solve
\begin{equation}
\begin{aligned} 
\label{bdDiff}
\partial_t \bd^{N,12}_\kappa 
=&
 \Delx \bd^{N,12}_\kappa
    -
    \tfrac{1}{\epsilon^2}( \vert \bd^{N,1}_\kappa\vert^2 - 1)\bd^{N,12}_\kappa
-
   \tfrac{1}{\epsilon^2}\big(\vert \bd^{N,1}_\kappa\vert^2 
   - \vert \bd^{N,2}_\kappa\vert^2\big)\bd^{N,2}_\kappa
   \\& -
  ( \overline{\bu^{N,12}_\kappa}\cdot \nabx) \bd^{N,1}_\kappa
  -
   ( \overline{\bu^{N,2}_\kappa}\cdot \nabx) \bd^{N,12}_\kappa  
\end{aligned}
\end{equation}
in $I_* \times  \Omega_{\zeta_\kappa}$ with zero initial and boundary conditions. 
 By testing \eqref{bdDiff} with $\Delx\bd^{N,12}_\kappa$,   we obtain
\begin{equation}
\begin{aligned}
\label{finiteDirectorx}  
\int_{\Omega_{\zeta_\kappa}}&  
\vert\nabx\bd^{N,12}_\kappa (t)\vert^2 \dx
  +
 \int_0^t\int_ {\Omega_{\zeta_\kappa}} \vert\Delx\bd^{N,12}_\kappa \vert^2  \dx\dt'
 \\&
 \lesssim 
\int_0^t
\int_{ \Omega_{\zeta_\kappa} }    
\frac{1}{\epsilon^2}(\vert \bd^{N,1}_\kappa\vert^2 - 1) \bd^{N,12}_\kappa \cdot \Delx\bd^{N,12}_\kappa
\dx\dt'
\\&
+
\int_0^t
\int_{ \Omega_{\zeta_\kappa} }
\frac{1}{\epsilon^2}\big(\vert \bd^{N,1}_\kappa\vert^2 
   - \vert \bd^{N,2}_\kappa\vert^2\big)\bd^{N,2}_\kappa 
 \cdot
 \Delx\bd^{N,12}_\kappa 
 \dx
\dt'
\\&
+
\int_0^t
\int_{ \Omega_{\zeta_\kappa} }
  ( \overline{\bu^{N,12}_\kappa}\cdot \nabx)  \bd^{N,1}_\kappa \cdot
 \Delx\bd^{N,12}_\kappa
 \dx
\dt'
\\&
+
\int_0^t
\int_{ \Omega_{\zeta_\kappa} }
  ( \overline{\bu^{N,2}_\kappa}\cdot \nabx)  \bd^{N,12}_\kappa \cdot
 \Delx\bd^{N,12}_\kappa
 \dx
\dt'
\end{aligned}
\end{equation}
for all $t\in I_*$.
Firstly, by using  the embedding $L^\infty(I_*;W^{1,2}(\Omega_{\zeta_\kappa}))\hookrightarrow L^\infty(I_*;L^6(\Omega_{\zeta_\kappa}))$, we have that
\begin{align*}
\int_0^t
\int_{ \Omega_{\zeta_\kappa} }    
\frac{1}{\epsilon^2}(\vert \bd^{N,1}_\kappa\vert^2 - 1)&\bd^{N,12}_\kappa \cdot \Delx\bd^{N,12}_\kappa
\dx\dt'
\lesssim
\delta
\int_{I_*}
\Vert \Delx\bd^{N,12}_\kappa\Vert_{L^2(\Omega_{\zeta_\kappa})}^2
\dt 
\\&
+
\big(1+
\sup_{t\in I_*}\Vert  \bd^{N,1}_\kappa\Vert_{W^{1,2}(\Omega_{\zeta_\kappa})}^4
\big)
\int_{I_*}
\Vert  \nabx\bd^{N,12}_\kappa\Vert_{L^2(\Omega_{\zeta_\kappa})}^2
\dt
\end{align*}
for any $\delta>0$. Similarly,  
\begin{align*}
\int_0^t
\int_{ \Omega_{\zeta_\kappa} }
\frac{1}{\epsilon^2}\big(\vert \bd^{N,1}_\kappa\vert^2 
   &- \vert \bd^{N,2}_\kappa\vert^2\big)
   \bd^{N,2}_\kappa \cdot \Delx\bd^{N,12}_\kappa
 \dx
\dt'
\\&
\lesssim
\int_0^t
\int_{ \Omega_{\zeta_\kappa} }
\frac{1}{\epsilon^2}\big(\vert \bd^{N,1}_\kappa\vert 
   + \vert \bd^{N,2}_\kappa\vert\big)\vert \bd^{N,12}_\kappa\vert
   \bd^{N,2}_\kappa \cdot \Delx\bd^{N,12}_\kappa
 \dx
\dt'
\\&
\lesssim
\delta
\int_{I_*}
\Vert \Delx\bd^{N,12}_\kappa\Vert_{L^2(\Omega_{\zeta_\kappa})}^2
\dt 
\\&
+
\big(\sup_{t\in I_*}\Vert  \bd^{N,1}_\kappa\Vert_{W^{1,2}(\Omega_{\zeta_\kappa})}^4
+
\sup_{t\in I_*} \Vert  \bd^{N,2}_\kappa\Vert_{W^{1,2}(\Omega_{\zeta_\kappa})}^4
\big)
\int_{I_*}
\Vert  \nabx\bd^{N,12}_\kappa\Vert_{L^2(\Omega_{\zeta_\kappa})}^2
\dt.
\end{align*} 
For the convective terms, we have
\begin{align*}
\int_0^t
\int_{ \Omega_{\zeta_\kappa} }
  ( \overline{\bu^{N,12}_\kappa}\cdot \nabx)  \bd^{N,1}_\kappa \cdot
 \Delx\bd^{N,12}_\kappa
 \dx
\dt'
\lesssim&
\delta
\int_{I_*}
\Vert \Delx\bd^{N,12}_\kappa\Vert_{L^2(\Omega_{\zeta_\kappa})}^2
\dt 
\\&
+
\int_{I_*}
\Vert \overline{\bu^{N,12}_\kappa}\Vert_{L^\infty(\Omega_{\zeta_\kappa})}^2
\Vert \nabx\bd^{N,1}_\kappa\Vert_{L^2(\Omega_{\zeta_\kappa})}^2
\dt 
\\
\lesssim&
\delta
\int_{I_*}
\Vert \Delx\bd^{N,12}_\kappa\Vert_{L^2(\Omega_{\zeta_\kappa})}^2
\dt 
\\&
+
\sup_{I_*}
\Vert \nabx\bd^{N,1}_\kappa\Vert_{L^2(\Omega_{\zeta_\kappa})}^2
\int_{I_*}
\Vert \overline{\bu^{N,12}_\kappa}\Vert_{L^2(\Omega_{\zeta_\kappa})}^2
\dt 
\end{align*}
by using the equivalence of all norms on $\mathrm{span}\{\bm{\psi}_1, \ldots\bm{\psi}_N\}.$
On the other hand, by using the definition $\overline{\bu^{N,2}_\kappa}:=\sum_{n=1}^N\overline{\alpha^{N,2}_n}\bm{\psi}_n$, we also have
\begin{align*}
\int_0^t
\int_{ \Omega_{\zeta_\kappa} }
  ( \overline{\bu^{N,2}_\kappa}\cdot \nabx)  \bd^{N,12}_\kappa \cdot
 \Delx\bd^{N,12}_\kappa
 \dx
\dt'
\lesssim&
\delta
\int_{I_*}
\Vert \Delx\bd^{N,12}_\kappa\Vert_{L^2(\Omega_{\zeta_\kappa})}^2
\dt 
\\&
+
R^2\sup_{1\leq n\leq N}
\Vert\bm{\psi}_n
\Vert_{L^\infty(\Omega_{\zeta_\kappa})}^2
\int_{I_*}
\Vert \nabx\bd^{N,12}_\kappa\Vert_{L^2(\Omega_{\zeta_\kappa})}^2
\dt .
\end{align*}
Now since
\begin{align}
\bd^{N,1}_\kappa,\bd^{N,2}_\kappa\in L^\infty\big(I_*; W^{1,2}(\Omega_{\zeta_\kappa}\big)
\cap 
L^2\big(I_*; W^{2,2}(\Omega_{\zeta_\kappa})\big),
\end{align}
if we now absorb the small $\delta$-terms into the corresponding term on the left-hand side and use Gr\"onwall's lemma, it follows that
\begin{equation}
\begin{aligned}
\label{finiteDirectorxX} 
\sup_{t\in I_*}
\int_{\Omega_{\zeta_\kappa}}   
\vert\nabx\bd^{N,12}_\kappa  \vert^2 \dx
  &+
 \int_{I_*}\int_{\Omega_{\zeta_\kappa}} \vert\Delx\bd^{N,12}_\kappa \vert^2  \dx\dt 
\\&\lesssim(1+T_*\exp(cT_*)) 
\int_{I_*} 
\Vert \overline{\bu^{N,12}_\kappa}\Vert_{L^2(\Omega_{\zeta_\kappa})}^2 
\dt  
\\&\lesssim(1+T_*\exp(cT_*))T_*
\bigg( 
\sup_{t\in I_*}
\Vert \overline{\bu^{N,12}_\kappa}\Vert_{L^2(\Omega_{\zeta_\kappa})}^2 
+
\int_{I_*} 
\Vert \nabx\overline{\bu^{N,12}_\kappa}\Vert_{L^2(\Omega_{\zeta_\kappa})}^2 
\dt 
\Bigg)
\end{aligned}
\end{equation}
If we now substitute \eqref{finiteDirectorxX} into \eqref{energyEstAxxY}, we obtain 
\begin{equation}
\begin{aligned}
\label{finiteDirectorxXX} 
\Vert ( \eta^N_\kappa , \bu^N_\kappa ) \Vert_{X^{I_*}}^2 
\leq 
\frac{1}{2}\Vert (\overline{\eta^N_\kappa},\overline{\bu^N_\kappa}) \Vert_{X^{I_*}}^2
\end{aligned}
\end{equation}
by choosing $T_*>0$ small enough.
This contraction property completes the proof of the existence of a unique local weak solution $(\eta^N_\kappa,\bu^N_\kappa,\bd^N_\kappa)$ for the fully coupled finite-dimensional system. 
The fact that this solution is global follows from the energy estimate. Similar to \eqref{energyEst}, we obtain
\begin{equation}
\begin{aligned}
\label{energyEstD1a}
\sup_{t\in I} 
&\int_\omega \big( \vert\partial_t \eta^N_\kappa \vert^2+ \vert\naby^2\eta^N_\kappa \vert^2\big)(t) \dy
+
\sup_{t\in I}
\int_{\Omega_{\zeta_\kappa}} \Big(\vert   \bu^N_\kappa   \vert^2
+
\vert\nabx\bd^N_\kappa  \vert^2
+
\tfrac{1}{2\epsilon^2}(\vert\bd^N_\kappa \vert^2 - 1)^2\Big)(t)
 \dx
\\& 
+
\int_I
\int_\omega \vert\partial_t\naby \eta^N_\kappa\vert^2\dy \dt
+
\int_I
 \int_{\Omega_{\zeta_\kappa}} \Big( \vert \nabx \bu^N_\kappa \vert^2
 +
 \big\vert\Delx\bd^N_\kappa - \tfrac{1}{\epsilon^2}(\vert \bd^N_\kappa \vert^2 - 1)\bd^N_\kappa\big\vert^2 
 \Big)
  \dx\dt
\\&
\lesssim 
\int_\omega \big( \vert\partial_t\eta^N_\kappa(0)\vert^2+ \vert\naby^2\eta^N_\kappa(0)\vert^2\big) \dy 
+ 
\int_{\Omega_{\zeta_\kappa(0)}} \Big(\vert   \bu^N_\kappa(0)  \vert^2
+
\vert\nabx\bd^N_\kappa(0)  \vert^2
+
\tfrac{1}{2\epsilon^2}(\vert\bd^N_\kappa(0) \vert^2-1)^2\Big) 
 \dx .
\end{aligned}
\end{equation} 
If we use the strong convergence of the data, 
then we obtain
\begin{align*}
\eta^N_\kappa \rightarrow \eta_\kappa 
&\qquad\text{in}\qquad \big(L^\infty(I;W^{2,2}(\omega )),w^* \big),
\\
\partial_t\eta^N_\kappa \rightarrow \partial_t\eta_\kappa 
&\qquad\text{in}\qquad \big(L^\infty(I;L^2(\omega)),w^*\big) \cap
\big( L^2(I;W^{1,2}(\omega)), w\big),
\\
\bu^N_\kappa \rightarrow\bu_\kappa 
&\qquad\text{in}\qquad \big(L^\infty(I;L^2(\Omega_{\zeta_\kappa})),w^*\big)
\cap
\big( L^2(I;W^{1,2}(\Omega_{\zeta_\kappa})), w\big), 
\\
\bd^N_\kappa \rightarrow \bd_\kappa &\qquad\text{in}\qquad \big(L^\infty(I;W^{1,2}(\Omega_{\zeta_\kappa})),w^*\big)
\cap
\big( L^2(I;W^{2,2}(\Omega_{\zeta_\kappa})),w\big). 
\end{align*}
Furthermore, by using lower semi-continuity of norms,we can use the convergences above to pass to limit in \eqref{galerkinweak1}, \eqref{approxDAloneX}  and
\eqref{energyEstD1a} to complete the proof of Theorem \ref{thm:linReg}. 
 \end{proof}

\subsection{The regularized nonlinear system} 
In the previous section, we passed to the limit $N\rightarrow\infty$ in the Galerkin parameter to get a   weak solution $(\eta_\kappa,\bu_\kappa,\bd_\kappa)$ of the linearized system posed on the given regularized spacetime geometry $I \times  \Omega_{\zeta_\kappa}$. In this section, we use a fixed-point argument to show that $(\eta_\kappa,\bu_\kappa,\bd_\kappa)$ is actually a  weak solution of the nonlinear system posed on the unknown regularized geometry, i.e.  
\begin{align}
\divx \bu_\kappa=0, \label{divfreekNon} 
\\
 \partial_t^2 \eta_\kappa -\partial_t\partial_y^2 \eta_\kappa +  \partial_y^4 \eta_\kappa =   f(\eta_\kappa,\bu_\kappa,\pi_\kappa,\bd_\kappa), \label{shellEqukNon}
\\
\partial_t \bu_\kappa  + (\bu_\kappa\cdot \nabx)\bu_\kappa
= 
  \divx\bT_\kappa(\bu_\kappa,\pi_\kappa,\bd_\kappa) , \label{momEqukNon}
\\
\partial_t \bd_\kappa + (\mathbf{u}_\kappa\cdot \nabx) \bd_\kappa
=
  \Delx \bd_\kappa -\tfrac{1}{\epsilon^2}(\vert \bd_\kappa\vert^2 - 1)\bd_\kappa 
\label{solutekNon}
\end{align}
on $I\times\Omega_{\eta_\kappa}\subset \mathbb R^{1+2}$ where
\begin{align*}
f(\eta_\kappa,\bu_\kappa,\pi_\kappa,\bd_\kappa)=-(\mathbb{T}(\bu_\kappa,\pi_\kappa,\bd_\kappa) \bn_{\eta_\kappa})\circ \bm{\varphi}_{\eta_\kappa}  \cdot \bn\,|\partial_{y } \bm{\varphi}_{\eta_\kappa} |,
\\
\mathbb{T}(\bu_\kappa,\pi_\kappa,\bd_\kappa)= \nabx \bu_\kappa +(\nabx \bu_\kappa)^\top 
-\pi_\kappa\mathbb{I}_{2\times2}
-  \nabx \bd_\kappa\odot \nabx\bd_\kappa
\end{align*}
A weak solution of \eqref{divfreekNon}-\eqref{solutekNon} is defined in analogy with Definition \ref{def:weakSolution}. Our main result now reads. 
 \begin{theorem}
\label{thm:linRegNon}
Let $\kappa>0$ be arbitrary.
For a dataset $( \eta_{\kappa,0}, \eta_{\kappa,\star}, \bu_{\kappa,0}, \bd_{\kappa,0} )$
that satisfies \eqref{datasetKappaLin}, there  exists a  
weak solution $(\eta_\kappa, \bu_\kappa, \bd_\kappa)$ of   \eqref{divfreekNon}--\eqref{solutekNon}.
\end{theorem} 
\begin{proof}
We note that the only differences between \eqref{contEqKappaLin}-\eqref{directorEqKappaLin} and the anticipated system \eqref{divfreekNon}-\eqref{solutekNon} consists of the linearisation
by the given velocity $\bv_\kappa$ (rather than $\bu_\kappa$) in the advection term, the stress tensor term on the right-hand side of \eqref{shellEqKappaLin} being transformed by a coordinate transform with respect of $\zeta_\kappa$ (rather than $\eta_\kappa$), and finally the full system posed on $\Omega_{\zeta_\kappa}$ (rather than $\Omega_{\eta_\kappa}$). The proof of Theorem \ref{thm:linRegNon} therefore follows from the construction of a fixed point $\mathtt{T}( \zeta_\epsilon, \bv_\epsilon)=(\eta_\epsilon,\bu_\epsilon)$ of a certain solution map $\mathtt{T}$. Unfortunately, since we are dealing with weak solutions and the anticipated system \eqref{divfreekNon}-\eqref{solutekNon} is nonlinear, we are unable to use a Banach fixed-point-type argument as we did in the previous section. Consequently, we resort to a fixed-point theorem for set-valued mappings which gives the existence (but not uniqueness) of a fixed-point  \cite[Theorem A.4]{lengeler2014weak}. 
For this, with a slight abuse of notation, we (still) consider the interval $I_*:=(0,T_*)$ where $T_*$ is to be chosen later. For a tubular neighbourhood $S_\alpha$ of $\partial\Omega$ with $\alpha\leq L$, we set
\begin{align*}
X:=C(\overline{I}_* \times \omega ) \otimes L^2(I_*;L^2(\Omega\cup S_\alpha))
\end{align*}
and define the ball
\begin{align*}
B_R^X=\big\{ ( {\zeta}_\kappa, { {\bv}}_\kappa)\in X \,:\, {\zeta}_\kappa(0)= {\eta}_{0,\kappa}, \quad   \Vert ( {\zeta}_\kappa, { {\bv}}_\kappa)\Vert_{X}\leq R \big\}
\end{align*}
for $R>0$ large enough and for fixed $\kappa>0$. Now let us consider the solution map $\mathtt{T}:B_R^X \subset X  \rightarrow 2^{B_R^X}$
defined by $\mathtt{T}( {\zeta}_\kappa, { {\bv}}_\kappa)=( {\eta}_\kappa, { {\bu}}_\kappa)$. The critical requirement for a fixed point is to show compactness of the map $\mathtt{T}$. Thus, for a sequence $(\bd^n_\kappa)$ solving  \eqref{approxDAlone}, we consider any sequence $( {\zeta}_\kappa^n, { {\bv}}_\kappa^n)_{n\in \mathbb{N}}\in B_R^X$ with $\mathtt{T}( {\zeta}_\kappa^n, { {\bv}}_\kappa^n)=( {\eta}_\kappa^n, { {\bu}}_\kappa^n)$ (the existence of such a solution map is guaranteed by  \eqref{galerkinweak1} and \eqref{energyEstAxx}). Consequently, due to the regularity of the dataset, we have in particular that 
\begin{equation}
\begin{aligned} 
\label{eneryEstSolvStruAlonesmallN}
\sup_{t\in I}\Big(
\Vert\partial_t \eta^n_\kappa\Vert_{L^2(\omega)}^2 
+  
\Vert\partial_y^2\eta^n_\kappa\Vert_{L^2(\omega)}^2
\Big) 
\lesssim 1
\end{aligned}
\end{equation}
uniformly in $n\in \mathbb{N}$. Given that in dimension $d=1$, the embedding $W^{2,2}(\omega)\hookrightarrow C^{1,\beta}(\omega)$ with $\beta<\frac{1}{2}$ is compact and the embedding $C^{1,\beta}(\omega) \hookrightarrow L^2(\omega)$ is continuous, it follows from Aubin-Lions lemma that
\begin{align}
\label{strongCon0}
\eta^n_\kappa \rightarrow \eta_\kappa
\qquad\text{in}\qquad C(\overline{I}_*;C^{1,\beta} (\omega)).
\end{align}
Also, just as in  \cite[Lemma 6.3]{muha2019existence},
we can use a reformulated Aubin-Lions lemma \cite[Theorem 5.1]{muha2019existence} and the existence of a smooth solenoidal extension operator \cite[Corollary 3.4]{muha2019existence} to also obtain
\begin{align}
 \partial_t\eta^n_\kappa \rightarrow  \partial_t\eta_\kappa
&\qquad\text{in}\qquad L^2(I_*;L^2( \omega)),\label{strongCon1}
\\
  \mathbb{I}_{\Omega_{\zeta^n_\kappa}} \bu^n_\kappa \rightarrow 
   \mathbb{I}_{\Omega_{\zeta_\kappa}}  \bu_\kappa
&\qquad\text{in}\qquad L^2(I_*;L^2(\Omega\cup S_\alpha))
\label{strongCon2}
\end{align}
which finishes the proof of compactness of $\mathtt{T}$. Consequently,  the map
 $\mathtt{T}$ posses a fixed point, i.e., there exists $({\eta}_\kappa, { {\bu}}_\kappa) \in B_R^X$ with  $\mathtt{T}( {\eta}_\kappa, { {\bu}}_\kappa)=( {\eta}_\kappa, { {\bu}}_\kappa)$. The fact that the solution is global follows from  Section \ref{sec:apriori}.
\end{proof}

\subsection{Limits of the regularised system}
\label{sec:limit}
We are now going to pass to the limit $\kappa\rightarrow\infty$ in the regularisation parameter to complete the proof of Theorem \ref{thm:main}. Due to Theorem \ref{thm:linRegNon} (and Section \ref{sec:apriori}), it follows that
\begin{align*}
\eta_\kappa \rightarrow \eta
&\qquad\text{in}\qquad  \big(L^{\infty}(I;W^{2,2}(\omega)), w^*  \big),
\\
\partial_t\eta_\kappa \rightarrow \partial_t\eta
&\qquad\text{in}\qquad
\big(L^{\infty}(I;L^{2}(\omega)),w^*  \big)  \cap \big(L^{2}(I;W^{1,2}(\omega)),w  \big),
\\
\mathbb{I}_{\Omega_{\eta_\kappa}} \bu_\kappa \rightarrow \mathbb{I}_{\Omega_{\eta}} \bu
&\qquad\text{in}\qquad
\big(L^{\infty} \big(I; L^{2}(\Omega\cup S_\alpha)\big),w^*\big)\cap \big(L^2\big(I;W^{1,2}_{\divx}(\Omega\cup S_\alpha) \big), w\big),
\\
\mathbb{I}_{\Omega_{\eta_\kappa}} \bd_\kappa \rightarrow \mathbb{I}_{\Omega_{\eta}} \bd
&\qquad\text{in}\qquad
\big(L^{\infty} \big(I; W^{1,2}(\Omega\cup S_\alpha)\big),w^*\big)\cap \big(L^2\big(I;W^{2,2}(\Omega\cup S_\alpha) \big), w\big).
\end{align*}
Furthermore, just as in \eqref{strongCon0}-\eqref{strongCon2}, we also obtain
\begin{align}
\eta^n_\kappa \rightarrow \eta_\kappa
&\qquad\text{in}\qquad C(\overline{I} ;C^{1,\beta} (\omega))
\\
 \partial_t\eta_\kappa \rightarrow  \partial_t\eta
&\qquad\text{in}\qquad L^2(I ;L^2( \omega)),\label{strongCon1x}
\\
  \mathbb{I}_{\Omega_{\eta_\kappa}} \bu_\kappa \rightarrow 
   \mathbb{I}_{\Omega_{\eta}}  \bu
&\qquad\text{in}\qquad L^2(I ;L^2(\Omega\cup S_\alpha)).
\label{strongCon2x}
\end{align} 
For the director field, since its regularity is strong enough for the equation to hold almost every in space, we can use the equation to conclude that 
\begin{align*} 
\mathbb{I}_{\Omega_{\eta_\kappa}} \partial_t\bd_\kappa \rightarrow \mathbb{I}_{\Omega_{\eta}} \partial_t\bd
&\qquad\text{in}\qquad
  \big(L^2\big(I;L^{2}(\Omega\cup S_\alpha) \big), w\big).
\end{align*}
Thus,  by Aubin--Lions lemma,  the following strong convergence also hold 
\begin{align} 
\label{strongDirecLim0}
\mathbb{I}_{\Omega_{\eta_\kappa}} \bd_\kappa \rightarrow \mathbb{I}_{\Omega_{\eta}} \bd
&\qquad\text{in}\qquad
 C\big(\overline{I};L^{2}(\Omega\cup S_\alpha) \big) 
\cap  L^2\big(I;W^{1,2}(\Omega\cup S_\alpha) \big). 
\end{align} 
The above convergence results allow us to pass to the limit in the weak formulation, the interface condition, and the energy inequality. The proof of Theorem \ref{thm:main} is now done.

\section{Convergence of the Ginzburg--Landau approximation}
As shown in \cite{guillen2002global} for the fixed domain, we wish to now show that the Ginzburg--Landau approximation of the simplified Ericksen--Leslie model by Lin \cite{lin1989nonlinear}  converges to the actual simplified Ericksen--Leslie model in our flexible geometry. More precisely, we want to show that if $(\eta_\epsilon,\bu_\epsilon,\bd_\epsilon)$ is the weak solution constructed earlier, then it converges, as $\epsilon\rightarrow0$, to some $(\eta,\bu,\bd)$ that solves the system
\begin{align} 
\divx \bu=0
&\qquad \text{in }I \times  \Omega_\eta,
\label{contEqX}
\\
 \partial_t^2\eta - \partial_t\naby^2 \eta + \naby^4 \eta = f(\eta,\bu,p,\bd) 
&\qquad \text{on }I \times \omega,
\label{shellEqX}
\\
 \partial_t \bu  + (\bu\cdot \nabx)\bu  
= 
\divx\mathbb{T}(\bu,\pi,\bd) 
&\qquad \text{in }I \times  \Omega_\eta,
\label{momEqX}
\\
\partial_t \bd + (\bu\cdot \nabx) \bd
=
  \Delx \bd +\vert \nabx\bd\vert^2  \bd 
&\qquad \text{in }I \times  \Omega_\eta,
\label{directorEqX}
\end{align}
with initial conditions 
\begin{align}
&\eta(0, \by) = \eta_0(\by), \quad \partial_t\eta(0, \by) = \eta_*(\by)
&\quad \text{on }  \omega,
\label{initialStructureX}
\\
&\bd(0, \bx)=\bd_0(\bx), \qquad\bu(0, \cdot) = \bu_0(\bx)
&\quad \text{in }  \Omega_{\eta_0}.
\label{initialDensityVeloX}
\end{align}
and boundary conditions   
\begin{align}
\label{noSlipX}
& \bu \circ \bm{\varphi}_\eta =\partial_t\eta\bn
&\quad \text{on }I \times \omega, 
\\&\partial_{\bn_\eta}\bd=0  
&\quad \text{on }I \times \partial\Omega_\eta.\label{dstarX}
\end{align} 
By a weak solution, we mean the following precise definition
\begin{definition}[Weak solution] \label{def:weakSolutionX}
Let $(\eta_0, \eta_*, \bu_0,\bd_0 )$ be a dataset such that
\begin{equation}
\begin{aligned}
\label{datasetX}
&
\eta_0 \in W^{2,2}(\omega) \text{ with } \Vert \eta_0\Vert_{L^\infty(\omega)}<L , 
\\
& 
\eta_* \in L^2(\omega),
\quad 
\bu_0\in L^2_{\mathrm{\Div}}(\Omega_{\eta_0}),
\\&
\text{and }\bu_0\circ\bm{\varphi}_{\eta_0} =\eta_*\bn\text{ on $\omega$}, 
\\&
\bd_0 \in W^{1,2}(\Omega_{\eta_0}),
\quad
\vert \bd_0\vert=1 \text{ a.e.  in } \Omega_{\eta_0}.
\end{aligned}
\end{equation} 
We call the triple
$(\eta,\bu,\bd)$
a weak solution to the system \eqref{contEqX}--\eqref{dstarX} with data $(  \eta_0,  \eta_*,\bu_0,\bd_0)$ provided that the following holds:
\begin{itemize}
\item[(a)] The structure displacement $\eta$ satisfies
\begin{align*}
\eta \in W^{1,\infty} \big(I; L^2(\omega) \big)\cap W^{1,2} \big(I; W^{1,2}(\omega) \big)\cap  L^\infty \big(I; W^{2,2}(\omega) \big) 
\quad \text{with} \quad \Vert \eta \Vert_{L^\infty(I \times \omega)} <L,
\end{align*}
as well as $\eta(0,\by)=\eta_0(\by)$ and $\partial_t\eta(0,\by)=\eta_*(\by)$.
\item[(b)] The velocity field $\bu$ satisfies
\begin{align*}
 \bu \in L^\infty \big(I; L^2(\Omega_{\eta}) \big)\cap  L^2 \big(I; W^{1,2}_{\Div}(\Omega_{\eta}) \big) \quad \text{with} \quad 
\bu \circ\bm{\varphi}_\eta = \partial_t\eta\bn\quad\text{on}\quad I\times\omega,
\end{align*}
as well as $\bu(0,\bx)=\bu_0(\bx)$. 
\item[(c)] The director field $\bd$ satisfies
\begin{align*}
 \bd \in L^\infty \big(I; W^{1,2}(\Omega_{\eta}) \big)\cap  L^2 \big(I; W^{2,2}(\Omega_{\eta}) \big)
 \quad \text{with} \quad 
\vert \bd\vert=1 \quad\text{a.e. in}\quad I\times \Omega_\eta,
\end{align*}
as well as $\bd(0,\bx)=\bd_0(\bx)$.
\item[(d)] For all  $(\phi, {\bfphi}, \bm{\psi}) \in C^\infty(\overline{I}\times\omega) \otimes C^\infty(\overline{I} ;C^\infty_{\Div}(\R^2)) \otimes C^\infty(\overline{I}\times\R^2) $ with $\phi(T,\cdot)=0$, ${\bfphi}(T,\cdot)=0$ and $\bfphi\circ\bfvarphi_\eta =\phi {\bn}$ on $I\times\omega$, we have
\begin{align*}
\int_I  \frac{\mathrm{d}}{\dt}&\bigg(\int_\omega \partial_t \eta \, \phi \dy
+
\int_{\Omega_\eta}\bu  \cdot {\bfphi}\dx
+
\int_{\Omega_\eta}\bd \cdot \bm{\psi} \dx
\bigg)\dt 
\\
&=
\int_I \int_\omega \big(\partial_t \eta\, \partial_t\phi-\partial_t\naby\eta\,\naby \phi
-
\naby^2\eta\,\naby^2 \phi \big)\dy\dt
\\&+\int_I  \int_{\Omega_\eta}\big(  \bu\cdot \partial_t  {\bfphi} + \bu \otimes \bu: \nabx{\bfphi} 
 -  
\nabx\bu:\nabx {\bfphi} 
+
\nabx \bd\odot \nabx\bd:\nabx \bfphi \big) \dx\dt
\\&
+\int_I  \int_{\Omega_\eta}\big( \bd\cdot \partial_t\bm{\psi} 
+
\bu\otimes \bd:\nabx \bm{\psi}
+
\Delx \bd\cdot \bm{\psi} +  \vert\nabx \bd\vert^2 \bd\cdot  \bm{\psi}
\big) \dx\dt.
\end{align*}
\item[(e)]  
the estimate 
\begin{equation}
\begin{aligned}
\label{energyEstX}
 \frac{1}{2}\int_\omega& \big( \vert\partial_t \eta \vert^2+ \vert\naby^2\eta \vert^2\big)(t) \dy
+ 
 \frac{1}{2}\int_{\Omega_\eta} \Big(\vert   \bu  \vert^2
+
\vert\nabx\bd  \vert^2 \Big)(t)
 \dx
\\& 
+
\int_0^t
\int_\omega \vert\partial_{t'}\naby \eta\vert^2\dy \dt'
+
\int_0^t
 \int_{\Omega_{\eta(t')}} \Big( \vert \nabx \bu \vert^2
 +
 \big\vert\Delx\bd+ \vert \nabx\bd \vert^2 \bd\big\vert^2 
 \Big)
  \dx\dt' 
\\&\leq
 \frac{1}{2}\int_\omega \big( \vert\eta_*\vert^2+ \vert\naby^2\eta_0\vert^2\big) \dy 
+ 
 \frac{1}{2}
\int_{\Omega_{\eta_0}} \big(\vert   \bu_0  \vert^2
+
\vert\nabx\bd_0  \vert^2 \big) 
 \dx 
\end{aligned}
\end{equation}
holds for all $t\in I$. 
\end{itemize}
\end{definition} 
\begin{remark}
Similarly to Section \ref{sec:apriori}, the energy inequality above is obtained by testing the director field equation \eqref{directorEqX} with $\Delx \bd +\vert \nabx\bd\vert^2  \bd $. However, by using the relation $\vert\bd\vert=1$, it follows that $\partial_t\vert\bd\vert^2=\nabx\vert\bd\vert^2=0$ and so a couple of terms vanishes when the left-hand side of \eqref{directorEqX} is tested.
\end{remark} 
Our second main result reads as follows.
\begin{theorem}
\label{thm:mainX}
For a dataset $( \eta_0, \eta_*, \bu_0,\bd_0 )$ satisfying \eqref{datasetX}, a weak solution $(\eta,\bu,\bd)$ exists.
\end{theorem} 
\begin{proof}
From Theorem \ref{thm:main},  for a data $(  \eta_0,  \eta_*,\bu_0,\bd_0)$ satisfying \eqref{dataset}, we have found a weak solution $(\eta^\epsilon,\bu^\epsilon,\bd^\epsilon)$ to the system \eqref{contEq}--\eqref{momEq}. In particular, this solution satisfy
\begin{equation}
\begin{aligned}
\label{energyEstZ}
\frac{1}{2}\int_\omega& \big( \vert\partial_t \eta^\epsilon \vert^2+ \vert\naby^2\eta^\epsilon \vert^2\big)(t) \dy
+ 
 \frac{1}{2}
 \int_{\Omega_{\eta^\epsilon}} \Big(\vert   \bu^\epsilon  \vert^2
+
\vert\nabx\bd^\epsilon  \vert^2
+
\tfrac{1}{2\epsilon^2}( \vert\bd^\epsilon \vert^2 - 1)^2\Big)(t)
 \dx
\\& 
+
\int_0^t
\int_\omega \vert\partial_{t'}\naby \eta^\epsilon\vert^2\dy \dt'
+
\int_0^t
 \int_{\Omega_{\eta^\epsilon}} \Big( \vert \nabx \bu^\epsilon \vert^2
 +
 \big\vert\Delx\bd^\epsilon - \tfrac{1}{\epsilon^2}(\vert \bd^\epsilon \vert^2 - 1)\bd^\epsilon\big\vert^2 
 \Big)
  \dx\dt' 
\\&\leq
 \frac{1}{2}\int_\omega \big( \vert\eta_*\vert^2+ \vert\naby^2\eta_0\vert^2\big) \dy 
+ 
 \frac{1}{2}
\int_{\Omega_{\eta_0}} \Big(\vert   \bu_0  \vert^2
+
\vert\nabx\bd_0  \vert^2
+
\tfrac{1}{2\epsilon^2}(\vert\bd_0 \vert^2 - 1)^2\Big) 
 \dx 
\end{aligned}
\end{equation}
for all $t\in I$. Since $\vert \bd^\epsilon_0\vert \leq1$ in $\Omega_{\eta_0}$, it follows that
 \begin{align}
 \label{energyEstZ1}
\frac{1}{\epsilon^2}(\vert \bd^\epsilon_0\vert^2-1)\leq 0
\end{align}
holds uniformly in $\epsilon$  and from the maximum principle, Section \ref{sec:maxPrin}, we also that
\begin{align}
\label{energyEstZ2}
\vert \bd^\epsilon\vert \leq1\qquad \text{in}\qquad I\times\Omega_{\eta^\epsilon}.
\end{align}
Furthermore, by the  relation $\frac{x^2}{2}\leq y^2+(x-y)^2$ (where $y:=\tfrac{1}{\epsilon^2}(\vert \bd^\epsilon \vert^2 - 1)\bd^\epsilon$) and \eqref{energyEstZ2}, the estimate 
\begin{equation}
\begin{aligned}
\label{energyEstZ3}
\frac{1}{2}\int_0^t\int_{\Omega_{\zeta^\epsilon}}\vert \Delx \bd^\epsilon\vert^2\dx\dt'
&\leq 
\int_0^t
\int_{\Omega_{\eta^\epsilon}}\vert\Delx\bd^\epsilon -  \tfrac{1}{\epsilon^2}(\vert \bd^\epsilon \vert^2 - 1)\bd^\epsilon\big\vert^2\dx\dt'
\end{aligned}
\end{equation}
hold.
Consequently, if follows from the estimates \eqref{energyEstZ}-\eqref{energyEstZ3} that the shell displacement satisfy
\begin{align*}
\eta^\epsilon \rightarrow \eta
&\qquad\text{in}\qquad  \big(L^{\infty}(I;W^{2,2}(\omega)), w^*  \big),
\\
\partial_t\eta^\epsilon \rightarrow \partial_t\eta
&\qquad\text{in}\qquad
\big(L^{\infty}(I;L^{2}(\omega)),w^*  \big)  \cap \big(L^{2}(I;W^{1,2}(\omega)),w  \big),
\end{align*}
and by Aubin--Lions lemma, the following strong convergence hold
\begin{align}
\label{strongetaLim}
\eta^\epsilon \rightarrow \eta
&\qquad\text{in}\qquad  C(\overline{I};C^{1,\beta} (\omega)), \qquad \beta<\frac{1}{2}.
\end{align}
For the fluid's velocity, the energy inequality \eqref{energyEstZ} give
\begin{align*} 
\mathbb{I}_{\Omega_{\eta^\epsilon}} \bu^\epsilon \rightarrow \mathbb{I}_{\Omega_{\eta}} \bu
&\qquad\text{in}\qquad
\big(L^{\infty} \big(I; L^{2}(\Omega\cup S_\alpha)\big),w^*\big)\cap \big(L^2\big(I;W^{1,2}_{\divx}(\Omega\cup S_\alpha) \big), w\big).
\end{align*}
Furthermore, by making $\partial_t \bu^\epsilon$ the subject in the momentum equation, we are able to bound the other terms uniformly in $\epsilon>0$ in a suitable negative Sobolev space. More precisely, with the help of  the energy inequality  and Ladyzhenskaya's inequality, we obtain
\begin{align*} 
\mathbb{I}_{\Omega_{\eta^\epsilon}} \partial_t\bu^\epsilon \rightarrow \mathbb{I}_{\Omega_{\eta}} \partial_t\bu
&\qquad\text{in}\qquad
  \big(L^2\big(I;W^{-1,2}(\Omega\cup S_\alpha) \big), w\big), 
\end{align*}
and thus, by Aubin--Lions lemma,
\begin{align*} 
\mathbb{I}_{\Omega_{\eta^\epsilon}} \bu^\epsilon \rightarrow \mathbb{I}_{\Omega_{\eta}} \bu
&\qquad\text{in}\qquad
 L^2\big(I;L^{2}(\Omega\cup S_\alpha) \big). 
\end{align*}
For the director field, the estimates \eqref{energyEstZ}-\eqref{energyEstZ3} give
\begin{align}  \label{dweakEp}
\mathbb{I}_{\Omega_{\eta^\epsilon}} \bd^\epsilon \rightarrow \mathbb{I}_{\Omega_{\eta}} \bd
&\qquad\text{in}\qquad
\big(L^{\infty} \big(I; W^{1,2}(\Omega\cup S_\alpha)\big),w^*\big)
\cap \big(L^2\big(I;W^{2,2}(\Omega\cup S_\alpha) \big), w\big).
\end{align}
so that similar to the velocity field, we can make $\partial_t \bd^\epsilon$ the subject in director field equation and obtain
\begin{align*} 
\mathbb{I}_{\Omega_{\eta^\epsilon}} \partial_t\bd^\epsilon \rightarrow \mathbb{I}_{\Omega_{\eta}} \partial_t\bd
&\qquad\text{in}\qquad
  \big(L^2\big(I;L^{2}(\Omega\cup S_\alpha) \big), w\big).
\end{align*}
Thus,  by Aubin--Lions lemma,  the following strong convergence hold 
\begin{align} 
\label{strongDirecLim}
\mathbb{I}_{\Omega_{\eta^\epsilon}} \bd^\epsilon \rightarrow \mathbb{I}_{\Omega_{\eta}} \bd
&\qquad\text{in}\qquad
 C\big(\overline{I};L^{2}(\Omega\cup S_\alpha) \big) 
\cap  L^2\big(I;W^{1,2}(\Omega\cup S_\alpha) \big). 
\end{align}
We also note  that the following can  be deduced from the  estimates \eqref{energyEstZ}-\eqref{energyEstZ3}
\begin{align*}   
\mathbb{I}_{\Omega_{\eta^\epsilon}}\Big( \Delx\bd^\epsilon - \frac{1}{\epsilon^2}(\vert \bd^\epsilon \vert^2 - 1)\bd^\epsilon\Big)
\rightarrow \mathbb{I}_{\Omega_{\eta}} \be
&\qquad\text{in}\qquad
\big(L^{2} \big(I; L^{2}(\Omega\cup S_\alpha)\big),w\big),  
\\
\mathbb{I}_{\Omega_{\eta^\epsilon}}  \nabx\bd^\epsilon \odot  \nabx\bd^\epsilon 
\rightarrow \mathbb{I}_{\Omega_{\eta}}\mathbb{M}
&\qquad\text{in}\qquad
\text{measure}.
\end{align*}
\subsection*{Identifying the quadratic forcing term}
By making use of the earlier convergences for the director field, we can appropriately identify the limits $\be$ and $\mathbb{M}$. To identify the latter one, we first define the following with respect to the Hazawa transform:
\begin{equation}
 \begin{aligned}\label{transHan}
\overline{\bd}^\epsilon:= \bd^\epsilon\circ\bm{\Psi}_{\eta^\epsilon-\eta}, \qquad  
\bm{\phi}^\epsilon=\overline{\bm{\phi}} \circ\bm{\Psi}_{\eta^\epsilon-\eta}^{-1},\qquad J_{\eta^\epsilon-\eta}=\mathrm{det}(\nabx \bm{\Psi}_{\eta^\epsilon-\eta}),
\\
\mathbb{A}_{\eta^\epsilon-\eta}:=( \nabx\bm{\Psi}_{\eta^\epsilon-\eta}^{-1}\circ\bm{\Psi}_{\eta^\epsilon-\eta})^\top\mathbb{B}_{\eta^\epsilon-\eta},\qquad
\mathbb{B}_{\eta^\epsilon-\eta}:=J_{\eta^\epsilon-\eta}\nabx\bm{\Psi}_{\eta^\epsilon-\eta}^{-1}\circ\bm{\Psi}_{\eta^\epsilon-\eta},
\end{aligned}
\end{equation}
where $\overline{\bm{\phi}}\in C^\infty(\overline{I}\times\R^2)$. Then by using the Hanzawa transform to map  $\Omega_{\eta^\epsilon}$ into $\Omega_{\eta}$, we obtain
\begin{align*}
\int_0^t  \int_{\Omega_{\eta^\epsilon}}
\nabx\bd^\epsilon \odot  \nabx\bd^\epsilon :\nabx \bm{\phi}^\epsilon  \dx\dt'
=&
\int_0^t  \int_{\Omega_{\eta }}
\mathbb{A}_{\eta^\epsilon-\eta}
\nabx\overline{\bd}^\epsilon \odot 
 J_{\eta^\epsilon-\eta}^{-1} \mathbb{B}_{\eta^\epsilon-\eta}\nabx\overline{\bd}^\epsilon :\nabx \overline{\bm{\phi}}   \dx\dt'
\\=&
\int_0^t  \int_{\Omega_{\eta }}
\big( 
\mathbb{A}_{\eta^\epsilon-\eta}
\nabx\overline{\bd}^\epsilon 
-
\nabx\bd\big)\odot  J_{\eta^\epsilon-\eta}^{-1} \mathbb{B}_{\eta^\epsilon-\eta}\nabx\overline{\bd}^\epsilon :\nabx \overline{\bm{\phi}}   \dx\dt'
\\&+
\int_0^t  \int_{\Omega_{\eta }}
\nabx\bd \odot \big( J_{\eta^\epsilon-\eta}^{-1} \mathbb{B}_{\eta^\epsilon-\eta}\nabx\overline{\bd}^\epsilon 
-
\nabx\bd\big):\nabx \overline{\bm{\phi}}  \dx\dt'
\\&+
\int_0^t  \int_{\Omega_{\eta }}
\nabx\bd \odot  
\nabx\bd :\big(\nabx \overline{\bm{\phi}}  -\nabx  \bm{\phi} \big) \dx\dt'
\\&+
\int_0^t  \int_{\Omega_{\eta }}
\nabx\bd \odot  
\nabx\bd : \nabx  \bm\phi  \dx\dt'
\\&=:I_1+\ldots+I_4.
\end{align*} 
where $\bm{\phi}$ is taken to be the uniform limit of $\bm{\phi}^\epsilon$ as $\epsilon\rightarrow0$. Notice that the Hanzawa transform $\bm{\Psi}_\zeta$, c.f. \eqref{Hanz}, converges to the identity map as $\zeta$ convergences to zero uniformly. Consequently, with \eqref{strongetaLim} in hand, the uniform limit $\bm{\phi}$ can be identified with $\overline{\bm{\phi}}$.

With the decomposition made above, we have
\begin{align*}
I_1&\lesssim
\int_0^t
\Vert 
\mathbb{A}_{\eta^\epsilon-\eta}
\nabx\overline{\bd}^\epsilon 
-
\nabx\bd\Vert_{L^2(\Oeta)}  \Vert \eta^\epsilon-\eta\Vert_{W^{1,4}(\omega)}
\Vert\nabx\overline{\bd}^\epsilon\Vert_{L^4(\Oeta)} 
\Vert\nabx \overline{\bm\phi}^\epsilon \Vert_{L^\infty(\Oeta)} \dt'
\\
&\lesssim 
\sup_{t\in I}  \Vert \eta^\epsilon-\eta\Vert_{W^{5/2,2}(\omega)}
\bigg(\int_I\Vert \overline{\bd}^\epsilon\Vert_{W^{2,2}(\Oeta)}^2\dt\bigg)^{1/2} 
\\
& \quad\times\bigg(\int_I
\Vert 
\mathbb{A}_{\eta^\epsilon-\eta}
\nabx\overline{\bd}^\epsilon 
-
\nabx\bd\Vert_{L^2(\Oeta)}^2
\Vert\nabx \overline{\bm\phi}^\epsilon \Vert_{W^{3,2}(\Oeta)}^2 \dt\bigg)^{1/2}.
\end{align*} 
By using \eqref{energyEstZ}-\eqref{strongetaLim}, \eqref{strongDirecLim} and the smoothness of the test function, it follows that
\begin{align*}
I_1\rightarrow0 \qquad\text{as}\qquad\epsilon\rightarrow0.
\end{align*}
Similarly, we can show that
\begin{align*}
I_2,I_3\rightarrow0 \qquad\text{as}\qquad\epsilon\rightarrow0 
\end{align*}
so that
\begin{align*}
\lim_{\epsilon\rightarrow0}
\int_0^t  \int_{\Omega_{\eta^\epsilon}}
\nabx\bd^\epsilon \odot  \nabx\bd^\epsilon :\nabx \bm{\phi}^\epsilon  \dx\dt'
=
\int_0^t  \int_{\Omega_{\eta }}
\nabx\bd \odot  
\nabx\bd : \nabx  \bm\phi  \dx\dt'. 
\end{align*}
This completes the identification of the quadratic nonlinear forcing term with respect to the director field.
\subsection*{Identifying the right-hand side of the director field.}
To identify $\be$ with $\Delx\bd+|\nabx \bd|^2\bd$,
we first note that due to the (strong) regularity of the director field obtained from the estimates \eqref{energyEstZ}-\eqref{energyEstZ3}, the equation
\begin{align}
\label{almostEv}
\partial_t \bd^\epsilon + (\mathbf{u}^\epsilon\cdot \nabx) \bd^\epsilon
=
 \Delx \bd^\epsilon -\frac{1}{\epsilon^2}(\vert \bd^\epsilon\vert^2 - 1)\bd^\epsilon
\end{align}
hold a.e. in $I\times\Omega_{\eta^\epsilon}$. Since $\bd\cdot\bd^\perp=0$, it follows that
\begin{align}
\label{almostEv1}
 {\bd^\epsilon}^\perp\big(\partial_t \bd^\epsilon\cdot{\bd^\epsilon}^\perp + (\mathbf{u}^\epsilon\cdot \nabx) \bd^\epsilon\cdot{\bd^\epsilon}^\perp
\big)
=
 {\bd^\epsilon}^\perp(\Delx \bd^\epsilon \cdot{\bd^\epsilon}^\perp)
\end{align}
a.e. $I\times\Omega_{\eta^\epsilon}$. Moreover, by Lemmas \ref{appen:lemma}, the equation \eqref{almostEv1} is equivalent to 
\begin{align}
\label{almostEv2}
\partial_t \bd^\epsilon + (\mathbf{u}^\epsilon\cdot \nabx) \bd^\epsilon
=
 \Delx \bd^\epsilon + \vert \nabx\bd^\epsilon\vert^2 \bd^\epsilon,
 \qquad
 | \bd^\epsilon|=1
\end{align}
a.e. $I\times\Omega_{\eta^\epsilon}$. Beside the term $\Delx \bd^\epsilon$, we can now use the strong convergence results to identify all other terms in \eqref{almostEv2} almost everywhere. The limit of this last term can be identified by making use of an analogy of the transform in \eqref{transHan}.  We get
\begin{align*}
\int_0^t  \int_{\Omega_{\eta^\epsilon}}
\Delx\bd^\epsilon \cdot \bm{\psi}^\epsilon  \dx\dt'
=&
-
\int_0^t  \int_{\Omega_{\eta^\epsilon}}
\nabx\bd^\epsilon :\nabx \bm{\psi}^\epsilon  \dx\dt'
\\=& -
\int_0^t  \int_{\Omega_{\eta }} 
\mathbb{A}_{\eta^\epsilon-\eta}
\nabx\overline{\bd}^\epsilon:\nabx\overline{\bm{\psi}}   \dx\dt' 
\\=& -
\int_0^t  \int_{\Omega_{\eta }} 
(\mathbb{A}_{\eta^\epsilon-\eta}
\nabx\overline{\bd}^\epsilon-\nabx\bd):\nabx\overline{\bm{\psi}}  \dx\dt' 
\\&\quad -
\int_0^t  \int_{\Omega_{\eta }} 
\nabx\bd :(\nabx\overline{\bm{\psi}} -\nabx \bm{\psi} ) \dx\dt' 
\\&\quad +
\int_0^t  \int_{\Omega_{\eta }} 
\Delx\bd \cdot \bm{\psi}  \dx\dt' 
\\=&J_1+J_2+J_3
\end{align*} 
The  strong convergence \eqref{strongDirecLim} and \eqref{strongetaLim} are now enough to pass to the limit $J_1,J_2\rightarrow0$ so that $J_3$ is identified as the limit of the Laplacian term. 
\subsection*{Conclusion}
By using the converges above, we obtain for all  $(\phi, {\bfphi}, \bm{\psi}) \in C^\infty(\overline{I}\times\omega) \otimes C^\infty(\overline{I} ;C^\infty_{\Div}(\R^2)) \otimes C^\infty(\overline{I}\times\R^2) $ with $\phi(T,\cdot)=0$, ${\bfphi}(T,\cdot)=0$ and $\bfphi\circ\bfvarphi_\eta =\phi {\bn}$ on $I\times\omega$, 
\begin{align*}
\int_I  \frac{\mathrm{d}}{\dt}&\bigg(\int_\omega \partial_t \eta \, \phi \dy
+
\int_{\Oeta}\bu  \cdot {\bfphi}\dx
+
\int_{\Oeta}\bd \cdot \bm{\psi} \dx
\bigg)\dt 
\\
&=
\int_I \int_\omega \big(\partial_t \eta\, \partial_t\phi-\partial_t\naby\eta\,\naby \phi
-
\naby^2\eta\,\naby^2 \phi \big)\dy\dt
\\&+\int_I  \int_{\Oeta}\big(  \bu\cdot \partial_t  {\bfphi} + \bu \otimes \bu: \nabx{\bfphi} 
 -  
\nabx\bu:\nabx {\bfphi} 
+
\nabx\bd \odot  
\nabx\bd:\nabx \bfphi \big) \dx\dt
\\&
+\int_I  \int_{\Oeta}\big( \bd\cdot \partial_t\bm{\psi} 
+
\bu\otimes \bd:\nabx \bm{\psi}
+
\Delx \bd\cdot \bm{\psi} +  \vert\nabx \bd\vert^2 \bd\cdot  \bm{\psi}
\big) \dx\dt.
\end{align*}
This finishes the proof.
\end{proof} 
\section*{Statements and Declarations} 
\subsection*{Funding}
This work has been partly supported by Grant ME 6391/1-1 (543675748) by the German Research
Foundation (DFG).
\subsection*{Author Contribution}
The author wrote and reviewed the manuscript.
\subsection*{Conflict of Interest}
The author declares that they have no conflict of interest.
\subsection*{Data Availability Statement}
Data sharing is not applicable to this article as no datasets were generated
or analyzed during the current study.
\subsection*{Competing Interests}
The author have no competing interests to declare that are relevant to the content of this article.


\end{document}